\documentclass[a4paper,11pt]{article}

\usepackage[a4paper,vscale=0.75]{geometry}

\usepackage[utf8]{inputenc}
\usepackage[
colorlinks=true,
linkcolor=blue,
anchorcolor=blue,
citecolor=blue,
urlcolor=blue,
plainpages=false,
pdfpagelabels
]{hyperref}

\usepackage{microtype}
\usepackage{graphicx}
\usepackage{mathtools}
\usepackage{amsmath}
\usepackage{mathrsfs}
\usepackage[amsmath,thmmarks,standard]{ntheorem}
\usepackage{amsfonts}
\usepackage[capitalize]{cleveref}
\usepackage{algorithm}
\usepackage[noend]{algpseudocode}
\usepackage{enumerate}
\usepackage{comment}

\usepackage[noblocks]{authblk}

\usepackage[normalem]{ulem}

\usepackage[marginpar]{todo}
\usepackage{xcolor}

\usepackage[numbers]{natbib}
\usepackage{t1enc}
\usepackage{txfonts}

\newcommand\RR{{\mathbb{R}}}

\newcommand\NN{{\mathbb{N}}}

\renewcommand{\L}{\mathscr L}
\newcommand{\peaks}{\mathscr{M}}
\newcommand{\I}{\mathscr{I}}

\theoremstyle{remark}

\begin{document}
\title{Persistence Barcodes versus  Kolmogorov Signatures:  \\ Detecting Modes of One-Dimensional Signals\thanks{Research partially supported by DFG FOR 916, Volkswagen Foundation, and the {\sc Toposys} project FP7-ICT-318493-STREP}} 

\date{January 27, 2015}

\author[1]{Ulrich Bauer}
\author[2,3]{Axel Munk}
\author[2]{Hannes Sieling}
\author[4]{Max Wardetzky}

\affil[1]{Technische Universität München (TUM)}%
\affil[2]{Institute for Mathematical Stochastics, 
University of G{\"o}ttingen}%
\affil[3]{Max Planck Institute for Biophysical Chemistry, G{\"o}ttingen}%
\affil[4]{Institute of Numerical and Applied Mathematics,  University of G{\"o}ttingen}%

\maketitle
\begin{abstract}
We investigate the problem of  estimating the number of modes (i.e., local maxima)---a well known question in statistical inference---and we show how to do so without presmoothing the data. 
To this end, we modify the ideas of persistence barcodes by first relating persistence values in dimension one to distances (with respect to the supremum norm) to the sets of functions with a given number of modes, and subsequently working with norms different from the supremum norm.
As a particular case we investigate the \emph{Kolmogorov norm}. We argue that this modification has certain statistical advantages. We offer confidence bands for the attendant \emph{Kolmogorov signatures}, thereby allowing for the selection of relevant signatures 
with a statistically controllable error.  As a result of independent interest, we show that \emph{taut strings} minimize the number of critical points for a very general class of functions. We illustrate our results by several numerical examples.
\end{abstract}
AMS subject classification: Primary 62G05,62G20; secondary 62H12\\

\section{Introduction}
Persistent homology~\cite{EdelesbrunnerHarer10,Edelsbrunner2002Topological} provides a \emph{quantitative} notion of the stability or robustness of critical values  of a (sufficiently nice) real valued function $f$ on a topological space: the persistence of a critical value is a lower bound on the amount of perturbation (in the supremum norm) required for its elimination. Persistence measures the life span of homological features in terms of the difference between \emph{birth} and \emph{death} of such features---according to the filtration of the underlying topological space that arises from the sublevel sets of $f$.
Birth and death of homological features of $f$ can be encoded in a barcode diagram, see~\cite{Edelsbrunner2002Topological}. In this article, we consider what we call \emph{persistence signatures}, defined as (half) the life span (or persistence) of critical values, i.e., 
persistence signatures correspond to (half) the lengths of persistence barcodes and (when properly ordered) give rise to a descending sequence 
\begin{align}\label{eq:persistence-signatures}
s_{0,\infty}(f) \geq s_{1,\infty}(f) \geq s_{2,\infty}(f)\geq \cdots \  ,
\end{align}
where we appropriately account for multiplicity of critical values. %
In our setup, $s_{0,\infty}(f)$ denotes the largest \emph{finite} persistence value of $f$, and we append the sequence by zeros beyond the smallest positive persistence value of $f$. 

We consider one dimensional signals $f:[0,1]\to \RR$. For the moment, to illustrate our results, let $X$ denote the space of piecewise constant real-valued functions on a (variable) equipartition of $[0,1]$. (Later in our exposition, we also consider more general function spaces.) Let $X_k \subset X$ denote the space of functions with at most $k$ \emph{modes}, i.e., local maxima, where we only count inner local maxima. %
Our point of departure is the observation that 
$$
s_{k,\infty}(f)=\mathrm{dist_\infty}(f, X_k) \ ,
$$
i.e., $s_{k,\infty}(f)$ equals the distance of $f$ to the space of functions with at most $k$ modes with respect to the sup norm. 
This follows from the combination of two facts. First, from the celebrated stability theorem in persistent homology~\cite{CohenSteiner2007Stability}, which asserts that 
\begin{align*}%
|s_{k,\infty}(f)-s_{k,\infty}(g)|\leq \|f-g\|_\infty \quad \text{for all $k\geq 0$} \ .
\end{align*}
Second, from the fact that in oder to eliminate all positive persistence signatures of $f:[0,1]\to \RR$ with value less or equal to $\delta$, it suffices to change $f$ by $\delta$ in the sup norm, see~\cite{Bauer2012Optimal}.\footnote{Note that this result does no longer hold in dimensions greater than two.}

The fact that persistence signatures correspond to distances (with respect to the sup norm) to sets of functions with at most $k$ modes leads us to considering norms different from the sup norm. Our motivation is to ask how signatures arising from different norms compare in a statistical sense. 
To this end, consider an arbitrary metric $d$ on $X$ and define the metric signatures 
$$
s_k(f):=\mathrm{dist}(f, X_k)\quad \text{with respect to $d$} \ .
$$
Then $(s_k(f))$ is an descending sequence as in~\eqref{eq:persistence-signatures}, see \cref{fig:signatures}.
\begin{figure}[t]
\center{
 \includegraphics[width= 0.4\columnwidth]{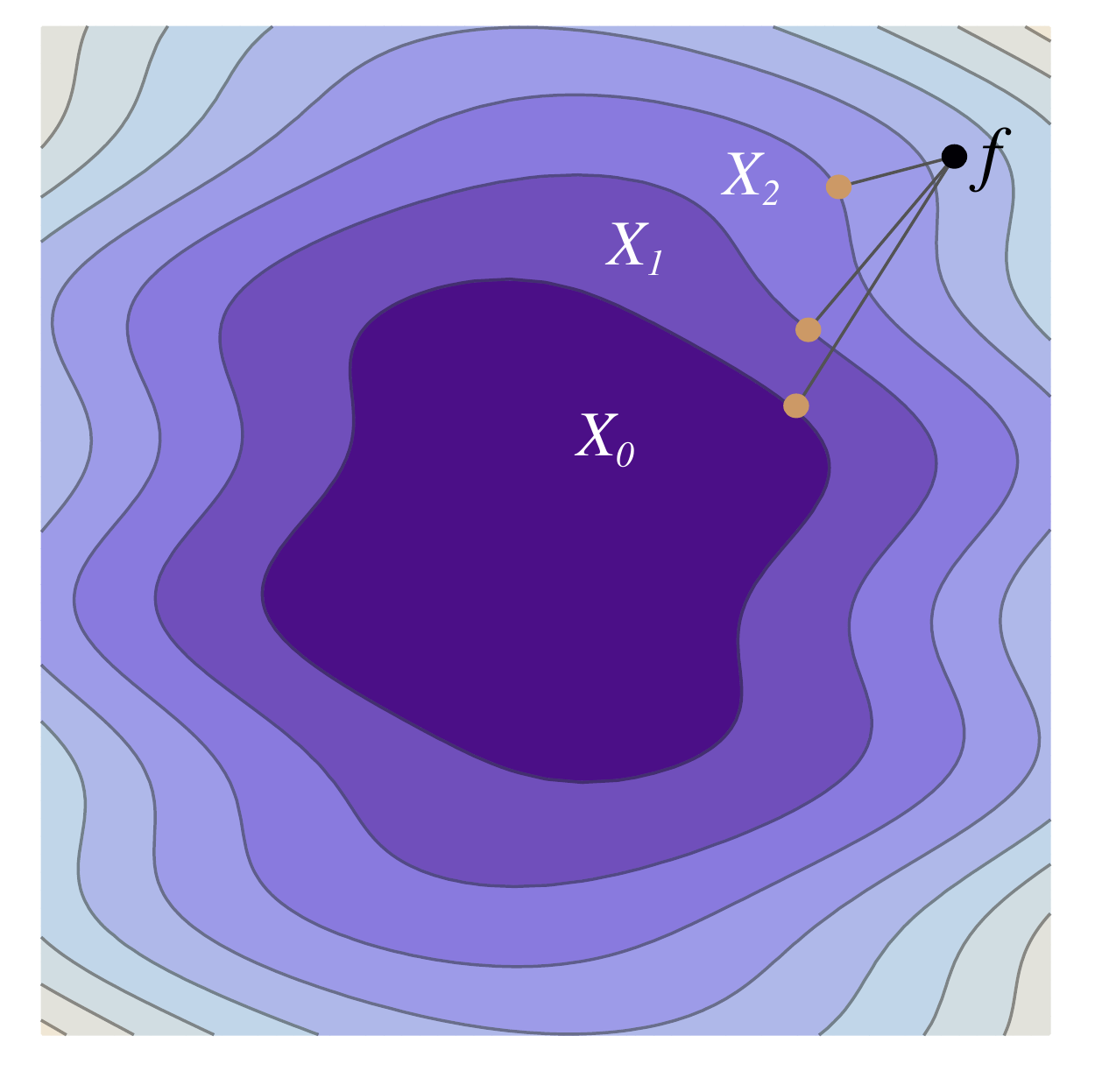}
 }
\caption{Illustration of metric signatures, i.e., distances of some $f\in X$ to the sets $X_k$ containing those functions with at most $k$ modes.}\label{fig:signatures}
\end{figure}
Moreover, since distance to sets in metric spaces is $1$-Lipschitz, stability is immediate:
\begin{align*}%
|s_k(f)-s_k(g)|\leq d(f,g) \quad \text{for all $k\geq 0$} \ .
\end{align*}
The resulting signatures $s_k(f)$ will in general be different from persistence signatures. The aim of this article is to analyze, from a statistical and algorithmic point of view, one particular example: the \emph{Kolmogorov metric} $d_K$ and its resulting Kolmogorov signatures $s_K$. For one dimensional signals $f:[0,1]\to \RR$ the Kolmogorov norm is defined as the $L^\infty$-norm of the antiderivative $F$ of $f$, subject to $F(0)=0$. The Kolmogorov norm plays a prominent role in probability and statistics, see, e.g.,~\cite{shorack2009empirical}. Our approach is based on the observation that if $s_k(f)=0$, then (the unknown function) $f$ has at most $k$ modes. This provides a link between \emph{mode hunting}, a widely studied problem in statistics~\cite{Harting85, Good80, Hartigan2000, Davies2004Densities, Silverman81}, and the robust estimation of signatures.  Most related to our approach is~\cite{Davies2004Densities}, where the Kolmogorov norm has been used for mode hunting in the context of density estimation.

In the sequel we consider the following basic statistical additive regression model. Suppose that $f:[0,1]\to \RR$ is corrupted by random noise $\epsilon$ and observed by a finite number of (equidistantly sampled) measurements $(Y_i)_{i=0}^n$, i.e., 
\begin{align}\label{intro:model}
Y_i = f(t_i) + \epsilon_i, \quad t_i = i/n \ .
\end{align}
Throughout we assume that the noise $(\epsilon_i)$ is independently distributed with mean zero such that for
some $\kappa>0$, $v>0$ and all $m\geq 2$,
\begin{align}\label{stat:ass:moments:intro}
 \mathbb{E}\left|\epsilon_i\right|^m \leq v m ! \kappa^{m-2}/2 \; \text{ for all} \;i=1,\ldots,n \ .
\end{align}

We are concerned with the following question: \emph{With what probability can one estimate the number of modes of $f$ (or provide bounds for its under- and overestimation) from the observations~$Y$?}

In dimension one, where mode hunting is intimately related to persistent homology, this question has been addressed in topological data analysis (TDA). A well known problem in this context is the fact that the stability theorem of persistent homology is based on the sup norm, which potentially makes this approach non-robust to outliers or unbounded noise. Therefore, 
several methods have been recently suggested to overcome this problem in various settings \citep{balakrishnan12minimax,Bendich2011Improving,Bubenik2007Statistical,Bubenik2010Statistical,Chazal2009GromovHausdorff,Chazal2011Geometric,Kloke2010Topological,Sheehy2012Multicover}. 
Roughly speaking, these methods have in common that they first regularize or filter the data in one form or another---in order to improve stability with respect to the sup norm---and then work with the persistence diagram of the so obtained preprocessed result. 
This is based on the initial estimation of $f$ itself. From a statistical perspective, however, having to estimate $f$ in a first step somewhat weakens the potential appeal of TDA.  
Already in dimension one of the underlying space, estimating $f$ by any regularization technique leads to difficult problems, e.g.,  data driven smoothing or parameter thresholding. We stress  that in addition, this sensibly affects the resulting persistence properties in a statistically hard to control manner, see, e.g., \citep{balakrishnan12minimax,Bubenik2010Statistical,fasy2014} for the case of a kernel estimator. In fact, presmoothing with a kernel estimator leads to what has been sometimes called the \emph{notorious bandwidth selection problem}, which 
does not posses a widely accepted solution since 
the optimal bandwidth (e.g., in the sense of minimizig the mean squared error between $f$ and its kernel estimate)---although theoretically known---depends on unknown characteristics  of $f$, such as its curvature (see \citep{wand95} among many others).  
Hence, we argue that a conceptual simplification and a computational advantage of TDA would result from circumventing explicit estimation of $f$. 

One aim of this paper is to show that direct estimation of topological properties of $f$ \emph{without} having to estimate 
$f$ itself is indeed a doable task by using Kolmogorov signatures. 
We confine ourselves to dimension one because using the Kolmogorov norm in this case  lends itself to an efficient algorithm ($O(n \log n)$ in the number of data points). We stress that that our statistical analysis carries over higher dimensions. 

A second aim of this paper is to provide confidence statements on the empirical Kolmogorov signatures with a controllable statistical error, similar in spirit to~\citep{fasy2014},  
where asymptotic  confidence bands for the empirical (sup norm based) persistence diagram are given for data on a manifold.
Their approach, however, is based on presmoothing for unbounded noise using a kernel density estimator, which we avoid in this paper.

\paragraph*{Inference for Kolmogorov signatures}
Using the Kolmogorov metric and the resulting Kolmogorov signatures, we investigate how well the empirical signatures $s_k(Y)$, obtained by interpreting $Y$ as a piecewise constant function, estimate the signatures $s_k(f)$. %
As a starting point, \cref{stat:thm:ineqbernst} asserts that under the moment condition~\eqref{stat:ass:moments:intro}, for any $\delta>0$ one has
\begin{align*}
 \mathbb P \left( \max_{k\in \NN_0} |s_k(Y)- s_k(f)| \geq \delta \right)  \leq 2 \exp\left( -  \frac{\delta^2 n}{2v +2\kappa \delta}\right) \ .
\end{align*}
Using this, \cref{stat:lem:confsign} asserts that for a given probability $\alpha\in (0,1)$, one can construct non-asymptotic confidence regions for the entire sequence $(s_k(f))$ of signatures in the sense that 
 \begin{align}\label{eq:conf-intro}
 \mathbb P \left( s_k(f) \in \left[\left(s_k(Y) - \tau_n(\alpha)  \right)_{+}, s_k(Y) + \tau_n(\alpha)  \right] \; \text{for all} \; k\in \mathbb{N}_0\right) \geq 1- \alpha \ ,
 \end{align}
where $(x)_+ = \max(0,x)$. Here $\tau_n(\alpha)$ depends in an explicit manner on $n$, $\alpha$, $\kappa$, and $v$, which are known constants or can be easily estimated from the data. We drop the dependence of $\tau_n$ on $\kappa$ and $v$ by considering $\kappa$ and $v$ fixed since we are mainly concerned with the dependence on $n$ and $\alpha$. For fixed $\alpha, \kappa, v$, one asymptotically has $\tau_n(\alpha)\approx 1/\sqrt{n}$. The parameter $\tau_n(\alpha)$ can be used to threshold the empirical signatures $s_k(Y)$ by defining 
\begin{align*}
k_{\epsilon} (Y) = \max\{k \in \mathbb{N}_0 : s_{k-1}(Y) \geq \epsilon \} \ ,
\end{align*} 
where, as a convention, we define $s_{-1}(Y)=\infty$. Then \cref{stat:thm:overest} asserts that for all $k \in \NN_0$, $f \in X_k$, and $\alpha\in (0,1)$, one has
\begin{align*}
 \mathbb P\left( k_{\tau_n(\alpha)} (Y) > k \right) \leq \alpha \ ,
 \end{align*}
i.e., the threshold parameter $\tau_n(\alpha)$ controls the probability of \emph{overestimating} the number of modes for any function $f\in X$. Notice that $\tau_n(\alpha)$ is independent of the number and magnitude of the modes of $f$, so in this sense the result is universal.
Obtaining a universal result in the other direction, i.e., controlling the probability of \emph{underestimating} the number of modes, is a more delicate task. Indeed, as pointed out in \citep{Donoho88}, obtaining such results is in general impossible if the modes of $f$ are allowed to become arbitrarily small. As a consequence, without a priori information on the ``smallest scales'' of $f$, no method can provide a control for their underestimation. Therefore, it is \emph{only} possible to provide a bound for underestimating those signatures of $f$ that are larger than a certain threshold.  \Cref{stat:thm:underest} asserts that 
for any $k \in \NN_0$, $f \in X_k$, and $\alpha \in (0,1)$, one has
\begin{align*}
 \mathbb P\left(  k_{\tau_n(\alpha)} (Y) < k_{2 \tau_n(\alpha)}(f) \right) \leq \alpha \ .
 \end{align*}
Combining the latter results, we obtain \emph{two sided bounds} for the estimated number of modes.
More precisely, for any $f\in X_k$ and any $\alpha \in (0,1)$ we obtain that 
\begin{align*}
  \mathbb P\left( k_{2 \tau_n(\alpha/2)}(f) \leq  k_{\tau_n(\alpha/2)} (Y) \leq k  \right) \geq 1-\alpha \ .
\end{align*}
As mentioned before, for fixed $\alpha, \kappa, v$, one has $\tau_n(\alpha)\approx 1/\sqrt n$. Therefore there exists a constant $C$ such that asymptotically (for large enough $n$) by thresholding at $C/\sqrt n$, it can be guaranteed at a level $\alpha$ that all signatures above this threshold are detected. Notice that so far we have not made use of any a priori information about $f$. This changes with \cref{stat:cor:muissspec}, which asserts that if $f \in X_k$ \emph{and} $s_{k-1}(f) \geq \epsilon$, then
\begin{align}\label{stat:eq:underest:intro}
  \mathbb P\left(  k_{\epsilon/2} (Y) = k \right) \geq 1 - 2 \exp \left(- \frac{\epsilon^2 n}{8 v + 4 \kappa \epsilon} \right) \ ,
\end{align}
i.e., the number of modes of $f$  can be estimated exponentially fast (in the number of samples) by thresholding the empirical signature \emph{provided that} one has a priori lower bounds on magnitude (in the Kolmogorov norm) of the smallest mode of $f$. Notice that this result is independent of the number of modes of $f$.

\paragraph*{Kolmogorov signatures vs.~persistence signatures} 
Kolmogorov signatures offer an alternative to persistence signatures, since they behave more robust for large errors $\epsilon_i$. The intuitive reason is that the Kolmogorov norm damps these errors, while they remain dominant using the sup norm without prefiltering. This is relevant, e.g., for unbounded noise (such as normally distributed errors, which are included in our noise model~\eqref{stat:ass:moments:intro}) or for data with outliers.

Nevertheless,  Kolmogorov signatures are not always superior to persistence signatures in terms of statistical efficiency. 
This can be seen by comparing their probabilities to detect a non vanishing signature from the data. 
To this end, we consider two limiting scenarios. The first comprises sparse signals with high peaks and small support, while the second  comprises weak signals with large support. To illustrate these scenarios, we consider 
functions with one \emph{single} mode and i.i.d.~normal errors with variance one, i.e., $\epsilon_i \sim \mathcal{N}(0,1)$. 

In the first scenario, we consider a sequence of functions 
\begin{align}\label{stat:subopt_Kolmogorov-stat-intro}
f_n(x)= \begin{cases}
 (1+\varepsilon)\sqrt{2\log n} \; \text{if} \; x\in[j/n,(j+1)/n)\ ,\\
 0 \; \text{otherwise}\ ,
 \end{cases}
\end{align}
for some $\varepsilon>0$ and for some $j\in \{0,\dots, n-1\}$ \emph{that is a priori not known}. We show in \cref{thm:Kolmogorov-suboptimal} that asymptotically (as $n\to \infty$) it is impossible to distinguish $f_n$ from the zero function by thresholding Kolmogorov signatures at $\tau_n(\alpha)$ as above. In contrast, for such signals, sup norm based thresholding of the vector $(Y_1, \dots, Y_n)$ is known to behave asymptotically minimax efficient in the sense of detecting a non vanishing mode with probability tending to one as $n\to \infty$, see, e.g.,~\cite{donoho2004,Ingster03}. Whether this efficiency carries over to persistence signatures is unknown to us.

In the second scenario, we consider a sequence of functions 
\begin{align}\label{stat:vanish_funct-intro}
f_n(x)= \begin{cases}
 \delta_n \; \text{if} \; x\in[1/3,2/3)\ ,\\
 0 \; \text{otherwise}\ ,
 \end{cases}
\end{align}
with $\delta_n \rightarrow 0$.
It is well known that it is possible to detect the single mode of $f_n$ with probability tending to one as $n\to \infty$ if $\delta_n \sqrt{n} \rightarrow \infty$, see, e.g., \cite{vanderVaart00}.
From~\eqref{stat:eq:underest:intro} it follows, using $\epsilon = \delta_n$, that Kolmogorov signatures can correctly detect the single mode of signals in~\eqref{stat:vanish_funct-intro} by thresholding signatures at $\delta_n/2$. In contrast, for persistence signatures, there exists no thresholding strategy that can detect the single mode with probability one. To be precise, let  again  $\delta_n \sqrt{n} \rightarrow \infty$,  and assume additionally $\delta_n \sqrt{\log n} \rightarrow 0$. Then \cref{stat:thm:supnorm} asserts that for an arbitrary sequence $(q_n)$ of reals one has
$
\limsup_{n \rightarrow \infty} \mathbb{P}\left(  k_{q_n} (Y) = 1 \right) < 1
$.

\paragraph*{Efficient computation using taut strings}
While our approach can in principle be extended to metric different from the ones induced by the sup or Kolmogorov norms, not every metric lends itself to an efficient computation of the requisite signatures. The difficulty is to compute the distance of a given function to the set of functions with at most $k$ modes. Using \emph{taut strings} (which are intimately related to total variation (TV) minimization~\citep{Grasmair2007Equivalence,Grasmair2008Generalizations, DavKov01, Mammen1997Locally}), we prove that the set of Kolmogorov signatures can be computed in $\mathcal{O}(n\log n)$ time, where $n$ is the number of observations. Given $f \in L^\infty([0,1])$ and $\alpha \geq 0$, the \emph{taut string}, $F_\alpha$, is the function whose graph has minimal total length (as a curve) among all absolutely continuous functions in the $\alpha$-tube around the antiderivative $F$ of $f$. Letting $f_\alpha=F_\alpha'$ denote the derivative of the taut string,  
\Cref{thm:tautStringMinModes} provides a result of independent interest  that has been implicitly used several times in the existing literature but has never been proven rigorously to our knowledge: $f_\alpha$ minimizes the number of modes among all $L^\infty$-functions in the (closed) $\alpha$-ball around $f$ with respect to the Kolmogorov norm. Indeed, our result generalizes previous results on the mode-minimizing property of $f_\alpha$, which were shown in the special context of piecewise constant functions using the Kolmogorov norm, see, e.g.,~\citep{DavKov01,Davies2004Densities,Hartigan2000,Mammen1997Locally}.

\section{Modes and signatures} 

\paragraph*{Modes}
Let $f:[0,1]\to \RR$ be an arbitrary function. In order to define the number of \emph{modes} (local maxima) of $f$, consider a finite partition $P=\{t_0, \dots, t_{|P|}\}$ of $[0,1]$ such that $0=t_0 <t_1<\dots<t_{{|P|}-1}<t_{|P|}=1$. For each $0<i<|P|$ let
\begin{align*}
\peaks(f,P,i) = 
\begin{cases} 
1 &\quad\text{if}\quad \mathrm{max}(f(t_{i-1}), f(t_{i+1}) ) < f(t_i) \\ 0 &\quad\text{else} \ .
\end{cases} 
\end{align*}
Define the number of modes of $f$ with respect to $P$ and the total number of modes of $f$ by
\begin{align*}
\peaks(f,P) =  \sum_{i=1}^{|P|-1} \peaks(f,P,i) \quad\text{and}\quad \peaks(f) = \sup_P \peaks(f,P) \ ,
\end{align*}
respectively. It is easy to see that if $f$ is constant, %
then $\peaks(f)=0$ and if $f$ is a Morse function in the classical sense (i.e., a smooth function with only nondegenerate critical points), then $\peaks(f)$ equals the (possibly infinite) number of local maxima of~$f$ on the open interval $(0,1)$. Notice that different from Morse theory, though, we are not concerned with critical values or critical points of functions; $\peaks(f)$ merely counts the \emph{number} of modes, without referring to their individual positions or values.  

\paragraph*{Metric signatures} 
We denote by $\L^\infty$ the linear space of Lebesgue-measurable essentially bounded functions on $[0,1]$. Notice that we do not regard $\L^\infty$ as a space of equivalence classes of functions. Throughout this article we work with functions in some (to be specified) set $X\subset \L^\infty$. For example, $X$ may consist of functions of bounded variation or piecewise polynomial functions. We do not a priori require $X$ to be a linear space. By $(X,d)$ we denote $X$ together with some metric, but we do not require $(X,d)$ to be a complete metric space. Additionally, we allow that $d$ attains the value $\infty$. Particular choices of $(X,d)$ will be specified below.

\begin{definition}[Metric signatures]
Let $X_k$ denote the subset of $X$ with at most $k$ modes, i.e., $X_k:= \{f\in X : \peaks(f)\leq k\}$.
Define the $k$th \emph{metric signature} of $f\in X$ as 
\begin{align*}
s_k(f) := d(f,X_k) = \inf_{g \in X_k} d(f,g) \quad\text{for}\quad k \in \NN_0 \ , 
\end{align*}
i.e., the distance of $f$ to the set of functions with at most $k$ modes. 
\end{definition}
Clearly, $X_k\subseteq X_{k+1}$ are nested models; hence, the sequence $(s_k(f))_{k\in \NN}$ is monotonically decreasing, %
and $s_k(f)$ measures the minimal distance by which $f$ needs to be moved (with respect to the metric $d$) in order to remove all but its $k$ most significant %
modes. What is considered significant and what is not, however, heavily depends on the choice of metric. 
In any case, so far we have not excluded pathologies, i.e., situations where $\peaks(f)>k$ but  $s_k(f)=d(f,X_k)=0$. Hence: %
\begin{definition}[Descriptive metric]
$(X,d)$ is called \emph{descriptive} if $\peaks(f)>k$ implies that $s_k(f)>0$ for every $f\in X$ and all $k\in \NN_0$.
\end{definition}
\paragraph*{Stability} Regardless of the concrete choice of metric, notice that distance to (arbitrary) sets in metric spaces is $1$-Lipschitz; therefore stability essentially comes for free:
\begin{lemma}[Stability of signatures]\label{intro:lem:stab}
Let $f, g \in X$. Then
\begin{align*}
|s_k(f)-s_k(g)|\leq d(f,g)
\end{align*}
for all $k \in \NN_0$.
\end{lemma}
Stability implies that a small perturbation of $f$ results in a small perturbation of the signatures $s_k(f)$.
\section{Persistence signatures and Kolmogorov signatures} 
In our setting a ``good'' metric is one that leads  to signatures that clearly separate significant modes (with respect to a given noise model) from insignificant ones. We investigate two choices.
\paragraph*{Persistence signatures}
One possible choice of metric is the one induced by the sup norm, i.e., $d_\infty(f,g)= \sup_x|f(x)-g(x)|$, which leads to signatures that have an interpretation in the context of persistent homology, as we show below. 
\begin{lemma}
$(X,d_\infty)$ is descriptive for every $X\subset \L^\infty$.
\end{lemma}
\begin{proof}
Being descriptive is equivalent to $X_k$ being closed in $X$ for all $k$. Suppose that there exists $k\in \NN_0$ such that $X_k$ is not closed, i.e., there exist $f \in X\setminus X_k$ and a sequence $(f_n)$ in $X_k$ with $d_\infty(f_n,f)\to 0$. Since $f \notin X_k$, there exists a partition $P=\{t_0, \dots, t_{|P|}\}$ of $[0,1]$  and some index set $I$ with $k<|I|<|P|$ such that $\mathrm{max}(f(t_{i-1}), f(t_{i+1}) ) < f(t_i)$ for all $i\in I$. Since $d_\infty(f_n,f)\to 0$, there exists $N\in \NN$ such that $\mathrm{max}(f_n(t_{i-1}), f_n(t_{i+1}) ) < f_n(t_i)$ for all $n\geq N$ and all $i\in I$. Contradiction. 
\end{proof}
The following lemma makes precise the relation between topological persistence and our notion of metric signatures for the sup norm. %
\begin{lemma}\label{lemma:persDiagam}
Let $X$ be a space of tame functions, i.e., $H_*(f^{-1}(-\infty,t])$ has finite rank for all $f\in X$ and all $t\in\RR$, and every $f\in X$ has a finite number of homologically critical values. 
Order the finite persistence values (counted with multiplicity) of some $f\in X$ according to their persistence, from highest to lowest, yielding a persistence sequence $(p_k(f))_{k\geq1}$. Using $d_\infty$ yields $p_k(f) = 2 s_{k-1}(f)$ for all $k\geq 1$.
\end{lemma}
\begin{proof}
Let $k\geq 1$. We first claim that $p_k(f) \leq 2 s_{k-1}(f)$. Let $(f_n)$ be a sequence in $X_{k-1}$ with $d_\infty(f_n,f) \leq s_{k-1}(f)+\frac 1 n$.  Notice that $p_k(g) = 0$ for all $g \in X_{k-1}\subset X$. By the stability theorem for persistence diagrams~\cite{CohenSteiner2007Stability}, one has $|p_k(g)-p_k(f)| \leq 2d_\infty(f,g)$ for all $f,g \in X$. Together these facts imply that
\begin{align*}
p_k(f) = |p_k(f) -p_k(f_n)|\leq 2d_\infty(f,f_n)\leq 2s_{k-1}(f) + \frac 2 n \ ,
\end{align*}
which proves the first claim.

To see that $p_k(f) \geq 2 s_{k-1}(f)$, observe that the bound provided by the stability theorem is tight in dimensions less or equal to $2$, see~\cite{Bauer2012Optimal}. Indeed, if $f$ is tame, then by moving $f$ by at most $\delta$ in the sup norm, it is possible to remove all its persistence pairs with persistence less or equal to $2\delta$ \emph{without} increasing the number of remaining persistence pairs. Hence there exists a function $g\in X_{k-1}$ with $d_\infty(g,f)\leq\frac12p_{k}(f)$, which implies that $s_{k-1}(f) = d_\infty(f,X_{k-1}) \leq d_\infty(f,g) \leq \frac12p_{k}(f)$.
\end{proof}

\paragraph*{Kolmogorov signatures}
For reasons that will become evident in the next section, we propose an alternative to persistence signatures, which we call \emph{Kolmogorov signatures}. Let $\L^1$ denote the space of Lebesgue-integrable functions on $[0,1]$. Due to compactness of $[0,1]$, we have that $\L^\infty \subset \L^1$.  The \emph{Kolmogorov distance}, $d_K$,  is defined as follows. Let $f, g \in \L^\infty$, and let $F,G$ denote the respective antiderivatives, where, as a convention, we require that  $F(0)=G(0)= 0$. Define 
\begin{align*}
d_K(f,g) := d_\infty(F,G)  \ .
\end{align*}
Notice that $d_K$ does not induce a metric on arbitrary subsets $X\subset \L^\infty$ since if $f=g$ almost everywhere (a.e.), then $d_K(f,g)=0$. Therefore, we work with a unique representative in each equivalence class of a.e.~identical functions by requiring that
\begin{align}\label{eq:uniqueRep}
X\subset \L:=\left\{f \in \L^\infty: f(t) = \lim_{\epsilon \to 0}\inf_{0<\delta<\epsilon} \frac{1}{t^+(\delta) -t^-(\delta)}\int_{t^-(\delta)}^{t^+(\delta)}f(s) \, ds \;\text{for all}\; t\in[0,1]\right\}  \ ,
\end{align}
where $t^-(\delta) = \max(0,t-\delta)$ and $t^+(\delta)=\min(1, t+\delta)$.

There indeed exists a (unique) representative in $\L$ for every equivalence class of a.e. identical functions in $\L^\infty$, since the right hand side of~\eqref{eq:uniqueRep} exists (and is finite) for all $t\in [0,1]$ and all $f\in \L^\infty$, and since Lebesgue's differentiation theorem asserts that every $f\in \L^1$ satisfies 
\begin{align*}
f(t) = \lim_{\delta \to 0} \frac{1}{2\delta}\int_{t-\delta}^{t+\delta}f(s) \, ds \quad\text{a.e.~on $(0,1)$} \ .
\end{align*}
We thus obtain a projection operator $\mathcal{P}:\L^\infty \to \L\subset \L^\infty$. Notice, however, that $\L$ is not a linear space, since $f\in \L$ does not necessarily imply that $-f\in \L$. %
Nonetheless, we may of course choose linear subspaces $X\subset \L$ for  specific applications. %

The following lemma further motivates our choice of $\L$.
\begin{lemma}\label{lemma:minModes}
For any class $[f]$ of a.e.~identical functions in $\L^\infty$, its unique representative $\mathcal{P}(f)\in \L$ minimizes the number of modes within that class. 
\end{lemma}
\begin{proof}
Let $f \in \L^\infty$ with representative $\tilde f:= \mathcal{P}(f)\in \L$. We show that $\peaks(\tilde f)\leq \peaks [f]$. Consider any finite partition $P$ of $[0,1]$ and assume that $t_i$ counts a mode of $\tilde f$, i.e., $\tilde f(t_i)-\max(\tilde f(t_{i-1}), \tilde f(t_{i+1})) > \epsilon$ for some $\epsilon>0$. Consider any open neighborhood $U_i$ of $t_i$. Since  
\begin{align*}
\tilde f(t_i) = \lim_{\epsilon \to 0}\inf_{0<\delta<\epsilon} \frac{1}{t_i^+(\delta) -t_i^-(\delta)}\int_{t_i^-(\delta)}^{t_i^+(\delta)}f(s) \, ds \ ,
\end{align*}
there must be some $t\in U_i$ with $f(t) \geq \tilde f(t_i) - \epsilon/2$. Since $U_i$ can be chosen arbitrarily small, there exists $t'_i$ arbitrarily close to $t_i$ such that  $f(t'_i) \geq \tilde f(t_i) - \epsilon/2$. By the same argument, there exist $t'_{i-1}$ and $t'_{i+1}$ arbitrarily close to $t_{i-1}$ and $t_{i+1}$, respectively, such that $f(t'_{i-1}) \leq \tilde f(t_{i-1}) + \epsilon/2$ and $f(t'_{i+1}) \leq \tilde f(t_{i+1}) + \epsilon/2$. By our choice of $\epsilon$ this implies $f(t'_i)>\max(f(t'_{i-1}), f(t'_{i+1}))$. Continuing this way yields a partition $P'$ with $\peaks(f,P') \geq \peaks(\tilde f, P)$.
\end{proof}
\begin{lemma}
$(X,d_K)$ is descriptive for every $X\subset \L$.
\end{lemma}

\begin{proof}
We show that $X\setminus X_k$ is open wrt.~the Kolomogorov metric. Let $f\in X\setminus X_k$. Then there exists a finite partition $P=\{t_0, \dots, t_{|P|}\}$ of $[0,1]$  and some index set $I$  with $k<|I|<|P|$ such that $f(t_i) - \max(f(t_{i-1}),f(t_{i+1}))>\epsilon$ for all $i\in I$ and some small enough $\epsilon>0$.  Without loss of generality, we assume that $t_{i-1} >0$ and $t_{i+1}<1$ for all $i\in I$. Let $\delta>0$ be small enough such that for all $i \in I$ the intervals $[t_{i-1}-\delta, t_{i-1}+\delta]$, $[t_{i}-\delta, t_{i}+\delta]$, and $[t_{i+1}-\delta, t_{i+1}+\delta]$ are contained in $[0,1]$ and are mutually disjoint. Additionally, for all  $i\in I$ and all $j \in \{-1,0,1\}$ let $\delta_{i,j}\leq \delta$ be such that 
\begin{align*}
\left |f(t_{i+j})-\frac{1}{2\delta_{i,j}}\int_{t_{i+j}-\delta_{i,j}}^{t_{i+j}+\delta_{i,j}} f(s) ds\right | < \frac{\epsilon}{4} \ .
\end{align*}
Let $\delta'= \min_{i\in I,  j \in \{-1,0,1\}} \delta_{i,j}$. Let $g\in X$ with $d_K(f,g) < \frac1 4 \epsilon \delta'$. Then 
\begin{align*}
\left |\int_a^b (f(s) -g(s)) ds\right | < \frac 1 2 \epsilon \, \delta'
\end{align*}
for all $0\leq a < b \leq 1$. Hence,
\begin{align*}
\left |f(t_{i+j})-\frac{1}{2\delta_{i,j}}\int_{t_{i+j}-\delta_{i,j}}^{t_{i+j}+\delta_{i,j}} g(s) ds\right | < \frac{\epsilon}{4} +\frac{\epsilon \, \delta'}{4\delta_{i,j}} \leq \frac{\epsilon}{2} 
\end{align*}
for all  $i\in I$ and all $j \in \{-1,0,1\} $.
Therefore, there exists $t_i' \in [t_i-\delta_{i,0}, t_i+\delta_{i,0}]$ with $g(t_i') > f(t_i)-\frac \epsilon 2$. Likewise, there exist $t_{i\pm 1}' \in [t_{i\pm 1}-\delta_{i\pm 1, \pm 1}, t_{i\pm 1}+\delta_{i\pm 1, \pm1}]$ with $g(t_{i\pm 1}') < f(t_{i\pm 1})+\frac \epsilon 2$.  Thus there exists a partition $P'$ of $[0,1]$ with $\peaks [g, P']$> k, i.e., $g \in X\setminus X_k$. Since $\epsilon$ and $\delta'$ only depend on $f$ and since $g$ was chosen arbitrarily in the open Kolmogorov-ball 
of radius $\frac1 4 \epsilon \delta'$ around $f$, this ball is contained in $X\setminus X_k$.
\end{proof}

\begin{figure}[t]
\center{
  \includegraphics[width= \columnwidth]{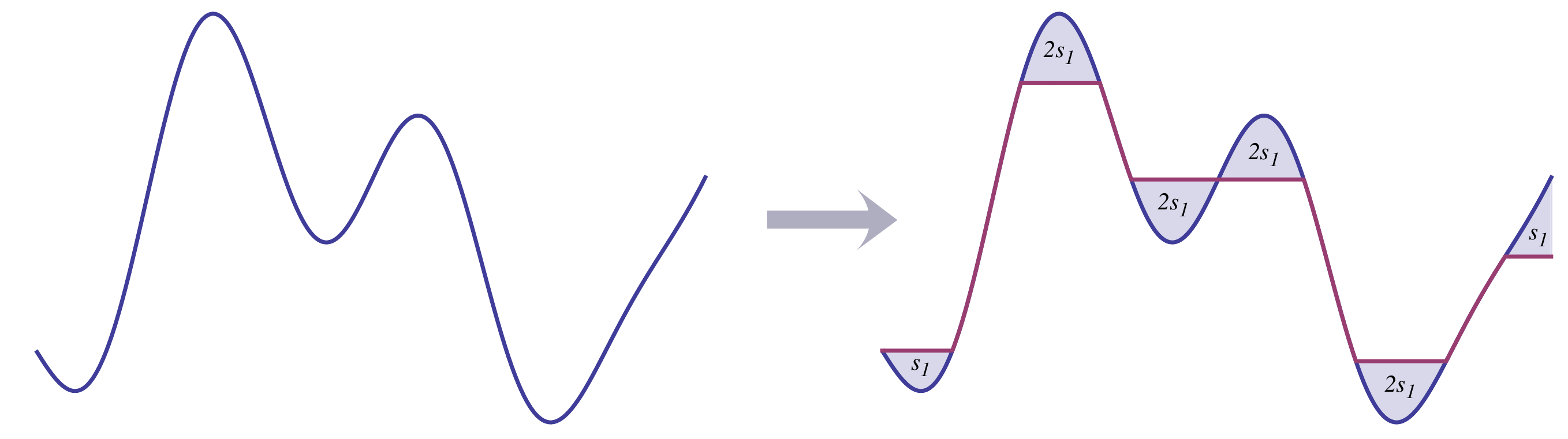}
 }
\caption{A function with exactly two modes (left) and its closest function with exactly one mode w.r.t.~the Kolmogorov norm (right, in purple). Notice that the attendant Kolmogorov signature, $s_1$, for removing the smallest mode of $f$, can be read off from the light-blue areas. The purple function is computed using \emph{taut strings} (see \cref{sec:taut-strings}).}\label{fig:Kolmogorov-vis}
\end{figure}

\Cref{fig:Kolmogorov-vis} offers a visualization for a function with two modes and its closest function with a single mode with respect to the Kolmogorov norm. Before elaborating on how to compute Kolmogorov signatures, though, we examine their statistical properties.

\section{Statistical perspective} \label{sec:stat}

Throughout this section we assume that the noise $(\epsilon_i)$ in Model \eqref{intro:model} is independently distributed with mean zero such that for
some $\kappa>0$, $v>0$ and all $m\geq 2$,
\begin{align}\label{stat:ass:moments}
 \mathbb{E}\left[\left|\epsilon_i\right|^m\right] \leq v m ! \kappa^{m-2}/2 \; \text{ for all} \;i=1,\ldots,n \ .
\end{align}
Distributions which satisfy \eqref{stat:ass:moments} include the centered normal distribution with variance $\sigma^2>0$, the (centered) Poisson distribution with intensity $\lambda$, or the Laplace distribution with variance $2 \lambda^2$. Moreover, any symmetric distribution around zero with compact support is covered by \eqref{stat:ass:moments}, including the uniform distribution
on an interval $[-B,B]$.

\subsection{Thresholding Kolmogorov signatures}

In this subsection we prove an exponential deviation inequality for the empirical Kolmogorov signatures (\cref{stat:thm:ineqbernst}), which allows us to construct uniform confidence bands for the unknown signatures $(s_{\!j\,}(f))_{j \in \mathbb N _ 0}$. More precisely, we provide a data dependent sequence of intervals $(I^{(n)}(\alpha)_j)_{j \in \mathbb N _0}$ that covers the (unknown) signatures with probability at least $1-\alpha$.

Let $f\in X_k\subset X\subset \L$, with $\L$ defined in \eqref{eq:uniqueRep}, have (an unknown number of) exactly $k$ modes. As stressed in the introduction, 
we do not aim at estimating the regression function $f$ itself but rather at inferring directly the sequence of signatures $s_{\!j\,}(f)$ together with the number of modes $k$ in such a way that estimates for these quantities can be provided at a prespecified error rate. 
This can be achieved by properly thresholding the sequence of empirical signatures.

In our analysis we consider equidistant sampling points $i/n$ and piecewise constant functions $f^{(n)}:[0,1]\to \RR$ defined as
\begin{align*}
f^{(n)}(t) = \sum_{i=0}^{n-1} \mathbf{1}_{\left[\frac{i}{n},\frac{i+1}{n}\right)}(t) f\left(\frac{i}{n}\right) \ .
\end{align*}
We define $X_j^{(n)}$ as the corresponding set of piecewise constant functions with at most $j$ modes, and we call $s_{\!j\,}(f^{(n)}) = \mathrm{dist}(f^{(n)}, X_j^{(n)})$ the \emph{quantized signature} of $f$.
Further, for the observation vector $Y=(Y_1,\ldots,Y_{n})$ we define the piecewise constant function
\begin{align*}
Y^{(n)}(t) = \sum_{i=1}^{n} \mathbf{1}_{\left[\frac{i-1}{n},\frac{i}{n}\right)}(t) Y_{i} \ .
\end{align*}
In the following, we call $s_{\!j\,}(Y^{(n)})$ the \emph{empirical signatures}.

\paragraph*{Function spaces}
In principle, the results of this subsection hold for any function space $X\subset \L$ as long as one can control the distance $d_K(f,f^{(n)})$ between $f$ and the quantized function $f^{(n)}$.
Accordingly, all subsequent results are formulated for the quantized signatures $s_{\!j\,}(f^{(n)})$. From those, the corresponding statements concerning $s_{\!j\,}(f)$ can be obtained along the following reasoning. Consider the (deterministic) approximation error between $f^{(n)}$ and $f$ in terms of the Kolmogorov metric
\begin{align}\label{stat:eq:approxerr}
 d_K(f,f^{(n)}) = \sup_{s\in[0,1]} \left| \int_{0}^{s} f(t)-f^{(n)}(t) dt \right| \ .
\end{align}
Then, due to \cref{intro:lem:stab} and the triangle inequality, it follows that
\begin{align*}
 \max_{j\in\mathbb{N}_0} |s_{\!j\,}(Y^{(n)}) -s_{\!j\,}(f)| \leq \max_{j\in\mathbb{N}_0} |s_{\!j\,}(Y^{(n)}) -s_{\!j\,}(f^{(n)})| + d_K(f, f^{(n)}) \ .
\end{align*}
Therefore, if $d_K(f, f^{(n)})$ is known, then the subsequent estimates on $|s_{\!j\,}(Y^{(n)}) -s_{\!j\,}(f^{(n)})|$ can readily be modified to obtain estimates on $|s_{\!j\,}(Y^{(n)}) -s_{\!j\,}(f)|$. E.g., if $f$ H\"older continuous, i.e., 
\begin{align*}
 \left| f(x) - f(y) \right| \leq C \left| x-y \right|^{\gamma} \quad \forall {(x,y)\in [0,1]}, \quad \gamma > 0 \ ,
\end{align*}
then 
\begin{align} \label{stat:eq:apprerr}
 d_K(f,f^{(n)}) \leq \frac{C}{\gamma+1} n^{-\gamma},
\end{align}
so that the approximation error is of order $n^{-\gamma}$. Hence, due to \cref{intro:lem:stab},
\begin{align*}
 \max_{j\in\mathbb{N}_0} |s_{\!j\,}(Y^{(n)}) -s_{\!j\,}(f)| \leq
\max_{j\in\mathbb{N}_0} |s_{\!j\,}(Y^{(n)}) -s_{\!j\,}(f^{(n)})|+ O (n^{-\gamma}) \ ,
\end{align*}
and all subsequent estimates and results can be modified accordingly.

\paragraph*{Statistical inference of signatures and modes without a priori information}
We return to our initial goal of providing tools for statistical inference on the signatures and modes. We start with investigating how well the empirical signatures $s_{\!j\,}(Y^{(n)})$ estimate the quantized signatures $s_{\!j\,}(f^{(n)})$. To this end, we control $d_K(f^{(n)},Y^{(n)})$ by the following exponential deviation bound, which is a direct consequence of \cite[Theorem B.2]{Rio2010Theorie}.
\begin{theorem} \label{stat:thm:ineqbernst}
Assume the moment condition in \eqref{stat:ass:moments}. Then, for any $\delta>0$ and any $f\in X$, one has
\begin{align*}
 \mathbb P \left( \max_{j\in \NN_0} |s_{\!j\,}(Y^{(n)})- s_{\!j\,}(f^{(n)})| \geq \delta \right)  \leq 2 \exp\left( -  \frac{\delta^2 n}{2v +2\kappa \delta}\right) \ .
\end{align*}
\end{theorem}

\begin{proof}
By stability of metric signatures (\cref{intro:lem:stab}), we have that 
\begin{align*}
 \mathbb P \left( \max_{j\in \NN_0} |s_{\!j\,}(Y^{(n)})- s_{\!j\,}(f^{(n)})| \geq \delta \right) \leq \mathbb P \left( d_K(f^{(n)},Y^{(n)}) \geq \delta \right) \ .
\end{align*}
Let $S_k=\sum_{i=1}^{k}\epsilon_i$ and observe that $d_K(Y^{(n)},f^{(n)})=\max_k |S_k|/n$. From \cite[Theorem B.2]{Rio2010Theorie} we obtain
\begin{align*}
 \mathbb P \left( d_K(Y^{(n)},f^{(n)}) \geq \delta \right) = \mathbb P \left( \max_k |S_k| \geq \delta n \right) \leq 2 \exp\left( - \frac{\delta^2 n}{2v +2\kappa \delta}\right) \ .
\end{align*}
\end{proof}

This results shows that the empirical signatures $s_{\!j\,}(Y^{(n)})$ are close to the quantized signatures $s_{\!j\,}(f^{(n)})$ with high probability \emph{simultaneously} for all $j\in\NN_0$.

\begin{figure}[t]
 \includegraphics[width= 0.45\columnwidth]{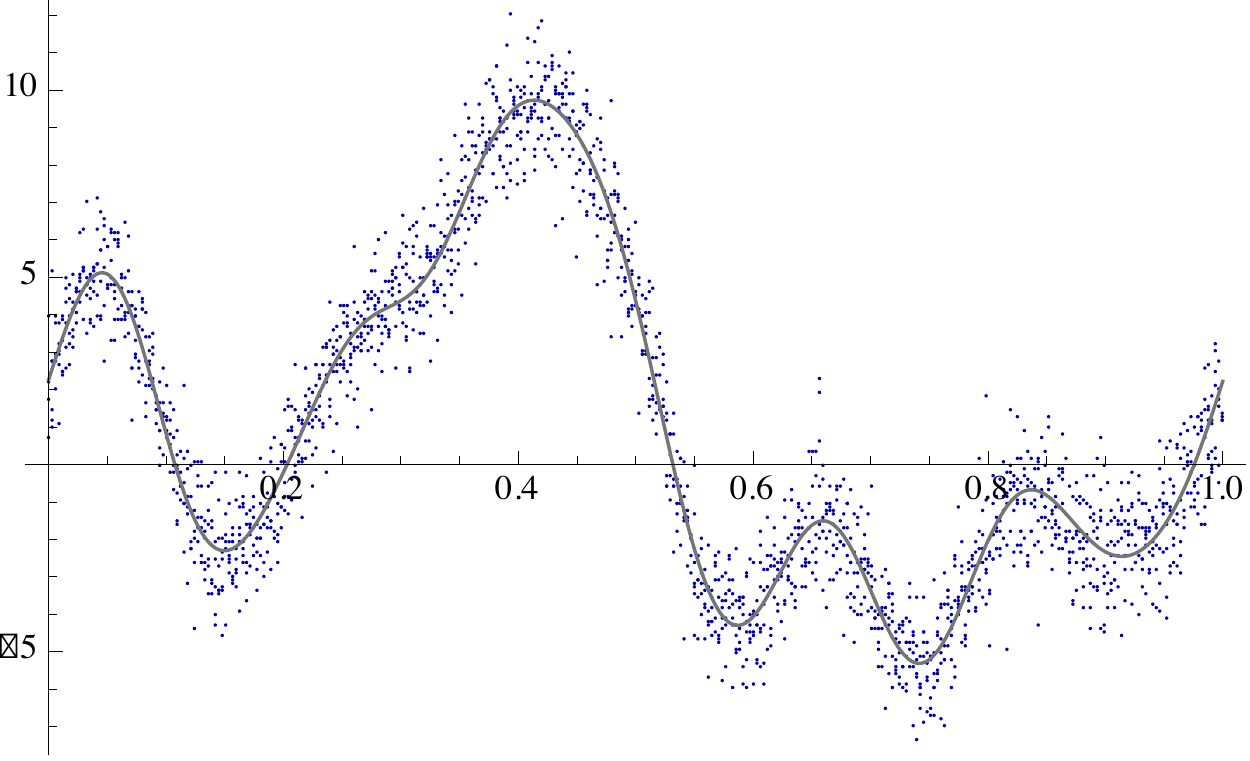}
 \quad
 \includegraphics[width= 0.45\columnwidth]{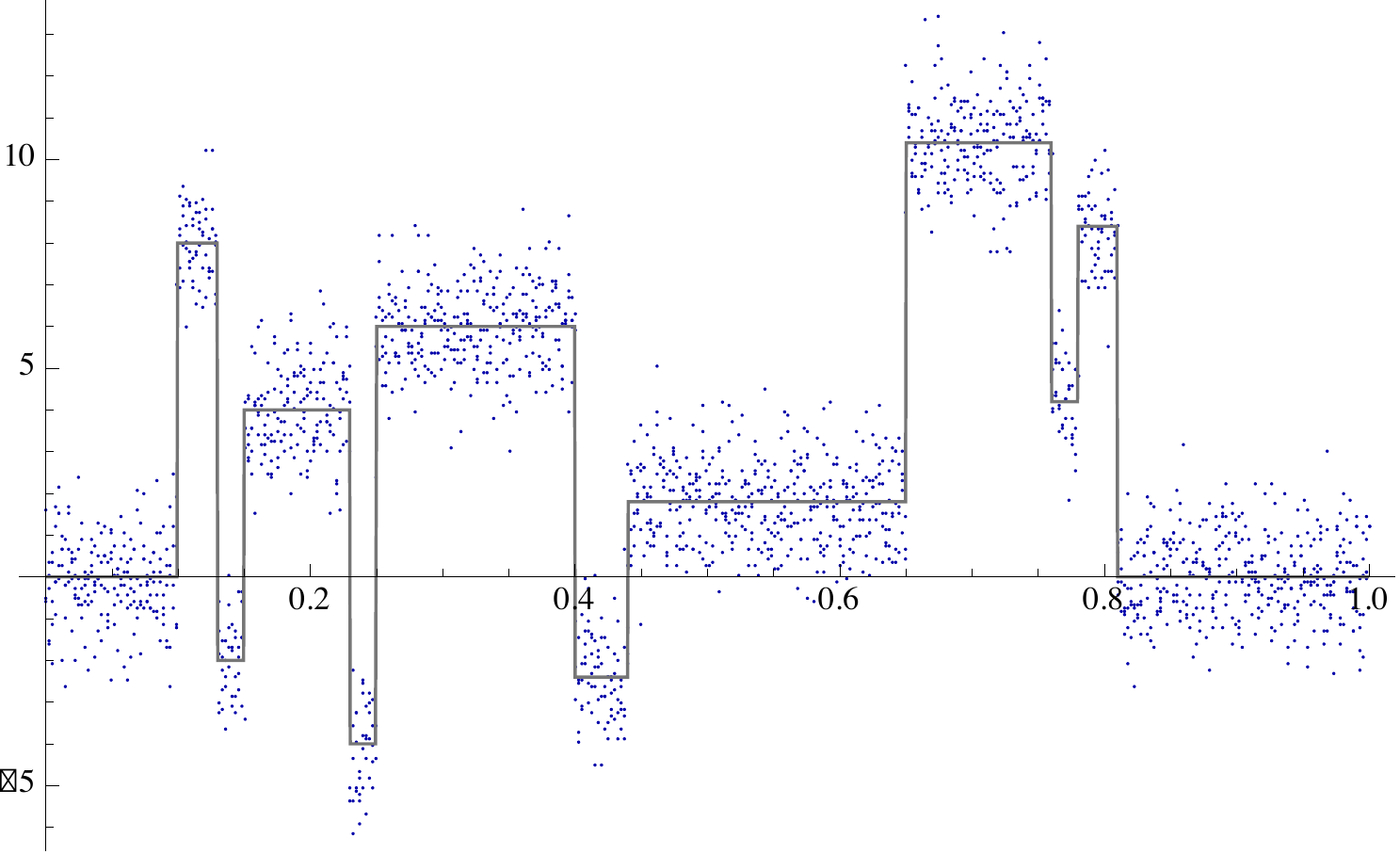}
 
 \bigskip
 
 \includegraphics[width= 0.45\columnwidth]{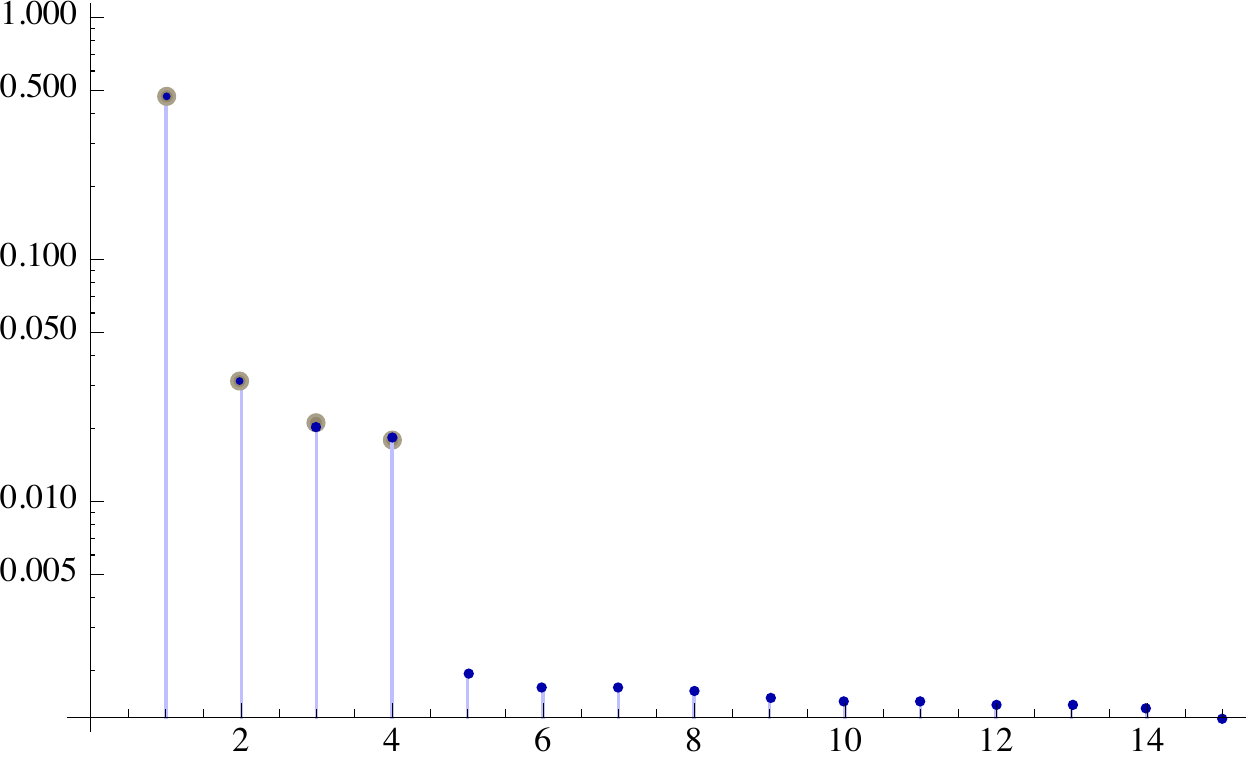}
 \quad
 \includegraphics[width= 0.45\columnwidth]{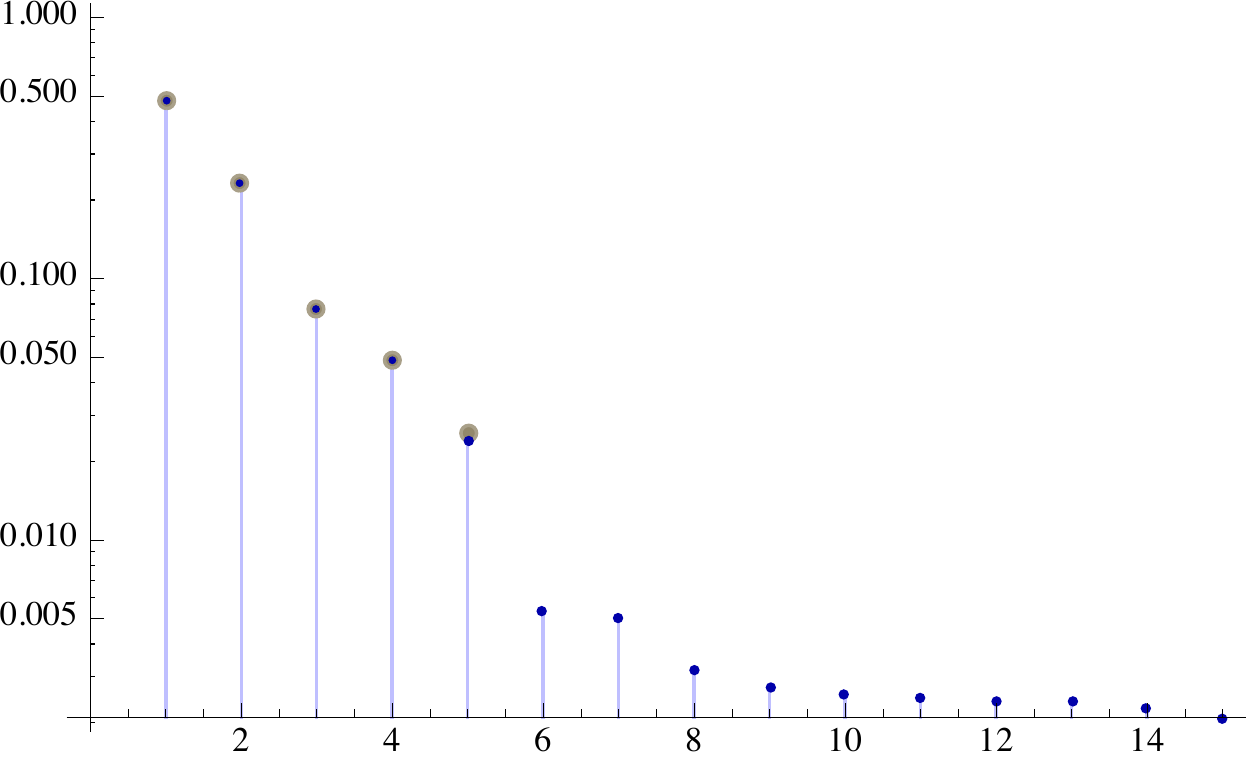}

 \caption{
 Noisy samples of signals (top) and signatures of both original signal and sample (bottom, $\log$-scale). Left: Function generated by random sampling and smoothing. Right: signal \texttt{bumps} \cite{Donoho1995Wavelet}. Sampling noise normally distributed with standard deviation $\sigma=1$. Notice that the largest signatures of signal and sample are very close (almost indistinguishable), and that there is a clear gap between the smallest signature of the signal (left: $k=4$, right: $k=5$) and the next signature of the noisy sample.
 }
 \label{stat:fig:ex}
\end{figure}
\begin{remark}[Sharpness of bound]
\Cref{stat:fig:ex} offers two examples of how the signatures of $Y^{(n)}$ deviate from those of $f^{(n)}$. Notice that in these examples, the signatures of $f^{(n)}$ are almost indistinguishable from the highest signatures of $Y^{(n)}$---indeed, their difference is \emph{less} than what is predicted by \cref{stat:thm:ineqbernst}. The reason is that, while the bound in \cref{stat:thm:ineqbernst} is sharp in general (since stability of metric signatures provides a sharp bound in general), it may be arbitrarily suboptimal for concrete examples, i.e., if $|s_k(f)-s_k(Y)|$ is small while $d_K(f,Y)$ is large.
\end{remark}

A useful application of \cref{stat:thm:ineqbernst} is that for a given probability $\alpha$, we can construct a non-asymptotic and honest (uniform) confidence region covering the signatures $s_{\!j\,}(f^{(n)})$ with probability at least $1-\alpha$, as shown in the following theorem.
\begin{theorem}\label{stat:lem:confsign}
 Fix some $\alpha\in(0,1)$ and let
  \begin{align*}
  \tau_n(\alpha):=\frac{1}{n}\left(\sqrt{\log(\alpha/2)\left(\log(\alpha/2)\kappa^2-2nv\right)}-\kappa \log(\alpha/2)\right) \ .
  \end{align*}
 Assume the regression model \eqref{intro:model} and the moment condition in \eqref{stat:ass:moments}. Then
 \begin{align*}
\inf_{f\in X} \mathbb P \left( s_{\!j\,}(f^{(n)}) \in \left[\left(s_{\!j\,}(Y^{(n)}) - \tau_n(\alpha)\right)_{+}, s_{\!j\,}(Y^{(n)}) + \tau_n(\alpha)\right] \; \text{for all} \; j\in \mathbb{N}_0\right) \geq 1- \alpha \ ,
 \end{align*}
 where $(x)_+ = \max(0,x)$.
\end{theorem}

\begin{proof}
From \cref{stat:thm:ineqbernst} we obtain
\begin{align*}
\mathbb{P} \left(|s_{\!j\,}(Y^{(n)}) -s_{\!j\,}(f^{(n)})| \leq \tau_n(\alpha) \; \text{for all} \; j\in \mathbb{N}_0 \right)
&\geq  1-2 \exp \left(- \frac{\tau_n(\alpha)^2 n}{2\nu + 2\kappa \tau_n(\alpha)} \right)\\
&= 1-\alpha \ .
\end{align*}
Since $s_{\!j\,}(f^{(n)})\geq 0$ for all $j\in\NN_0$, this completes the proof.
\end{proof}

Note that $\tau_n(\alpha)$ is a quantity that only depends on the values $n, \kappa, v$, and the confidence level~$\alpha$.
Here we assume for simplicity that $\kappa$ and $\nu$ are known---and while in practice this might not be the case, these numbers can be estimated
from the data, e.g., in the case of a normal distribution, such an estimate boils down to estimating the variance $\sigma^2$.
Fixing $\alpha$, we obtain a (random) sequence of intervals 
\begin{align*}
\left[\left(s_{\!j\,}(Y^{(n)}) - \tau_n(\alpha)\right)_{+}, s_{\!j\,}(Y^{(n)}) + \tau_n(\alpha)\right] \ ,
\end{align*}
which, according to \cref{stat:lem:confsign}, cover the sequence of true quantized signatures $s_{\!j\,}(f^{(n)})$ with confidence level $1-\alpha$.
For smaller values of $\alpha$, i.e., for larger confidence, these intervals become wider. %
Notice that for a fixed error $\alpha\in (0,1)$, the interval lengths $2\tau_n(\alpha)$ behave like $ 1/ \sqrt{n}$ as $n\rightarrow \infty$.

\Cref{stat:thm:ineqbernst} shows that $s_{\!j\,}(Y^{(n)})$ approximates $s_{\!j\,}(f^{(n)})$ well in the sup norm. However, the number of estimated signatures greater than zero might still be large. Consequently, $s(Y^{(n)})$ does not directly indicate which signatures are significantly larger than zero and hence will be of limited use for estimating the number of modes of $f$.  Nonetheless, such an estimate can readily be obtained by thresholding the empirical signatures. Define
\begin{align}\label{stat:def:est_nrmodes}
 k_{\epsilon} (Y^{(n)}) = \max\{j \in \mathbb{N}_0 : s_{j-1}(Y^{(n)}) \geq \epsilon \} \ ,
\end{align}
where, as a convention, we define $s_{-1}(Y^{(n)})=\infty$.

The threshold parameter $\tau_n(\alpha)$ has an immediate statistical interpretation:
It controls the probability of overestimating the number
of modes for any function $f\in X$.
\begin{theorem}\label{stat:thm:overest}
 Let $f\in X$,
 assume the regression model \eqref{intro:model} and the moment condition~\eqref{stat:ass:moments}, let $\alpha\in (0,1)$, and
 let $k \in \NN_0$ be such that $f^{(n)}\in X_k$. Then 
 \begin{align*}
 \mathbb P\left( k_{\tau_n(\alpha)} (Y^{(n)}) > k \right) \leq \alpha \ .
 \end{align*}
\end{theorem}

\begin{proof}
First, observe from the definition of $ k_{\tau_n(\alpha)}(Y^{(n)})$ in \eqref{stat:def:est_nrmodes} that
 \begin{align*}
\mathbb P\left(  k_{\tau_n(\alpha)}(Y^{(n)}) > k  \right) = \mathbb  P\left( s_{k}(Y^{(n)}) \geq \tau_n(\alpha)\right) \ .
\end{align*}
Notice that $f^{(n)}\in X_k$ implies that $s_k(f^{(n)})=0$. Therefore, for $f^{(n)}\in X_k$, \cref{stat:thm:ineqbernst} and the definition of $\tau_n(\alpha)$ imply (similar to the proof of \cref{stat:lem:confsign}) that 
 \begin{align*}
 \mathbb  P\left( s_{k}(Y^{(n)})  \geq \tau_n(\alpha)\right) \leq \alpha \ .
\end{align*}
\end{proof}

Hence, whatever the number of modes of $f^{(n)}$ might be, the thresholding index $ k_{\tau_n(\alpha)}(Y^{(n)})$ overestimates this number with probability less or equal to $\alpha$. Notice that the thresholding parameter $\tau_n(\alpha)$ is independent of the number and magnitude of the modes of $f$, so in that sense, this result is universal.

As mentioned in the introduction, obtaining a universal result in the other direction, i.e., controlling the probability of \emph{underestimating} the number of modes, is a more delicate task since modes can become arbitrarily small. Recalling the definition of $k_{\epsilon} (f)$ as in \eqref{stat:def:est_nrmodes}, we find:

\begin{theorem}\label{stat:thm:underest}
 Let $f\in X$,
 assume the regression model \eqref{intro:model} and the moment condition~\eqref{stat:ass:moments}, let $\alpha\in (0,1)$, and
 let $k \in \NN_0$ be such that $f^{(n)}\in X_k$. Then 
\begin{align}\label{stat:eq:overest}
 \mathbb P\left(  k_{\tau_n(\alpha)} (Y^{(n)}) < k_{2 \tau_n(\alpha)}(f^{(n)}) \right) \leq \alpha \ .
 \end{align}
\end{theorem}

\begin{proof}
Let $f^{(n)} \in X_k$ and let $l$ denote the largest integer such that $s_{l-1}(f^{(n)}) \geq 2 \tau_n(\alpha)$, i.e., $l= k_{2 \tau_n(\alpha)}(f^{(n)})$. If $l=0$, then \eqref{stat:eq:overest} is trivially satisfied, since $ k_{\tau_n(\alpha)} (Y^{(n)}) \geq 0$. So suppose that $l>0$. Then 
\begin{align*}
  \mathbb P\left(  k_{\tau_n(\alpha)}(Y^{(n)}) < l \right) &= \mathbb{P}\left( s_{l-1}(Y^{(n)}) < \tau_n(\alpha) \right) \\ 
  &\leq  \mathbb{P}\left( s_{l-1}(f^{(n)})-s_{l-1}(Y^{(n)}) >\tau_n(\alpha) \right)\\
  &\leq  \mathbb{P}\left( |s_{l-1}(f^{(n)})-s_{l-1}(Y^{(n)})| >\tau_n(\alpha) \right)\\
  &\leq \alpha \ ,
\end{align*}
where the last inequality follows from \cref{stat:thm:ineqbernst} and the definition of $\tau_n(\alpha)$.
\end{proof}
We have thus expressed the underestimation error of the number of modes as an explicit function of the signature threshold $2\tau_n(\alpha)$.
Combining the latter results, we obtain \emph{two sided bounds} for the estimated number of modes.
More precisely, for any $f$ and $k$ with $f^{(n)}\in X_k$ and any $\alpha \in (0,1)$ we have that 
\begin{align*}
  \mathbb P\left( k_{2 \tau_n(\alpha/2)}(f^{(n)}) \leq  k_{\tau_n(\alpha/2)} (Y^{(n)}) \leq k  \right) \geq 1-\alpha \ .
\end{align*}
As mentioned above, for fixed $\alpha, \kappa, v$ one has $\tau_n(\alpha)\approx 1/\sqrt n$. Therefore there exists a constant $C$ such that asymptotically (for large enough $n$) by thresholding at $C/\sqrt n$, it can be guaranteed at a level $\alpha$ that all signatures above this threshold are detected.

Based on the previous results we now construct confidence intervals for $k_\epsilon(f^{(n)})$, i.e., for the number of modes whose signatures exceed a certain size $\epsilon$.
\begin{corollary}\label{stat:corr:confint}
Assume the regression model \eqref{intro:model}, the moment condition~\eqref{stat:ass:moments}, let $\epsilon \geq 0$, and let $f^{(n)}\in X_k$.
Define
\begin{align*}
   l(\alpha,\epsilon)=\begin{cases}
                       \max \left\{j\in \NN_0 : s_{\!j\,}(Y^{(n)}) > \epsilon + \tau_n(\alpha)\right\} \; &\text{if} \; \epsilon < s_0(Y^{(n)})-\tau_n(\alpha) \\
                       0 \; & \text{otherwise}
                      \end{cases}
\end{align*}
 and 
 \begin{align*}
  u(\alpha,\epsilon) = \begin{cases}
                       \min \left\{j\in \NN_0 : s_{\!j\,}(Y^{(n)}) < \epsilon - \tau_n(\alpha)\right\} \; &\text{if} \; \epsilon > \tau_n(\alpha) \\
                       \infty \; & \text{otherwise}.
                      \end{cases}
\end{align*}
Then
\begin{align*}
  \mathbb P\left( k_\epsilon(f^{(n)}) \in \left[  l(\alpha,\epsilon) ,  u(\alpha,\epsilon) \right] \right) \geq 1- \alpha \ .
 \end{align*}
\end{corollary}

\begin{proof}
Suppose, for the moment, that
\begin{align}\label{proof:cor:confint:ass1}
 d_K(Y^{(n)}, f^{(n)}) \leq \tau_n(\alpha) \ .
\end{align}
Since $s_{\!j\,}(f^{(n)}) \geq \epsilon$ for all $j < k_\epsilon(f^{(n)})$, stability of metric signatures implies that $s_{\!j\,}(Y^{(n)}) \geq \epsilon -\tau_n(\alpha)$ for all $j < k_\epsilon(f^{(n)})$. Hence, by the definition of $ u(\alpha,\epsilon)$, we have $ u (\alpha,\epsilon) \geq k_\epsilon(f^{(n)})$.

Further, while still assuming \eqref{proof:cor:confint:ass1}, $\smash{s_{\!j\,}(Y^{(n)}) > \epsilon + \tau_n(\alpha)}$ implies that $\smash{s_{\!j\,}(f^{(n)})>\epsilon}$. 
Hence, by the definition of $ l(\alpha,\epsilon)$, we find that $\smash{s_{ l (\alpha,\epsilon)}(f^{(n)})>\epsilon}$.
This in turn implies $\smash{k_\epsilon(f^{(n)}) \geq  l (\alpha, \epsilon).}$
Therefore, we have so far shown that \eqref{proof:cor:confint:ass1} implies that
\begin{align*}
  l (\alpha,\epsilon) \leq k_\epsilon (f^{(n)}) \leq  u(\alpha, \epsilon) \ .
\end{align*}
Since \eqref{proof:cor:confint:ass1} holds with probability $\geq1-\alpha$ (see proof of \cref{stat:thm:ineqbernst}), this proves the assertion.
\end{proof}

Note that the upper bound for $k_\epsilon$ jumps to $\infty$ if $\epsilon \leq \tau_n(\alpha)$. This reflects the fact,
that meaningful upper bounds cannot be provided for signatures whose size is of the order of the noise level.

\begin{remark}[Distribution of signatures]\label{distrsign}
Assume the setting of \cref{stat:thm:ineqbernst} and suppose that $X_k$ is scaling invariant for all $k \in \mathbb{N}_0$, i.e., $\left\{ \lambda g : g \in X_k \right\} = X_k$ for all $0<\lambda \in \mathbb{R}$. %
Assume for simplicity that $f \equiv 0$, the general case still being unknown.
Then, for any $k \in \mathbb{N}_0$, we have that 
\begin{align*}
 \sqrt{n} \left(s_k(Y^{(n)})-s_k(f^{(n)})\right) & = \sqrt{n} \inf_{g \in X_k} \sup_{s\in [0,1]} \left| \int_0^s \left( \epsilon^{(n)}(t) - g(t) \right) dt \right| \\
 & =\inf_{g \in X_k} \sup_{s\in [0,1]} \left| \frac{1}{\sqrt n} \sum_{i=1}^{\lceil ns \rceil} \epsilon_i - \int_0^s  \sqrt{n}  g(t) dt \right| \\
 & =\inf_{g \in X_k} \sup_{s\in [0,1]} \left| \frac{1}{\sqrt n} \sum_{i=1}^{\lceil ns \rceil} \epsilon_i - \int_0^s g(t) dt \right| \ ,
\end{align*}
where the last equality follows from the scaling invariance of $X_k$. Noting that $$ \frac{1}{\sqrt{n}}  \max_{m =1, \dots, n}\left |\sum_{i=1}^m \epsilon_i \right| \overset{\mathcal D}{\rightarrow} \sup_{0\leq x\leq 1}B(x) \  ,$$ where $B$ denotes a standard Brownian motion on $[0,1]$ and using that $f\equiv 0$, it follows that 
\begin{align*}
 \sqrt{n} \left (s(Y^{(n)})\right ) \overset{\mathcal D}{\rightarrow} s(B') \ ,
\end{align*}
where $B'$ denotes the derivative of a standard Brownian Motion on $[0,1]$ in a weak sense. This follows from the continuity of the functional $s$ w.r.t.~the  Kolmogorov norm.
\end{remark}

\begin{remark}[Gaussian observation] \label{stat:rem:gauss}
If the noise $\epsilon$ in \eqref{intro:model} is Gaussian with mean zero and variance~$\sigma^2$, then \cref{stat:thm:ineqbernst} can be sharpened, due to a refined large deviation result for Gaussian observations (see, e.g.,~\cite{Billingsley99}):
\begin{align}
 \mathbb P \left( d_K(f^{(n)},Y^{(n)}) \geq \delta \right) \leq 2 \exp\left( -  \frac{\delta^2 n}{2 \sigma ^2}\right) \ .
\end{align}
Hence, in the Gaussian case, all results of \cref{sec:stat} remain true if $\tau_n(\alpha)$ is replaced by the simpler (and slightly sharper) threshold
\begin{align*}
 \tilde \tau_n(\alpha)= \sqrt{- 2 \sigma^2 / n \log (\alpha/2)} \ .
\end{align*}
\end{remark}

\paragraph*{Obtaining the correct number of modes using a priori information}
Notice that so far we have not made any a priori assumption about $f^{(n)}$. If, however, $f^{(n)}\in X_k$, and \emph{if} we impose prior information on the smallest strictly positive signature $s_{k-1}(f^{(n)})$, then we obtain an explicit bound for the probability that the number of modes is estimated correctly.

\begin{theorem}\label{stat:cor:muissspec}
Assume the regression model \eqref{intro:model} and the moment condition~\eqref{stat:ass:moments}. Let $f^{(n)} \in X_k$ be such that $s_{k-1}(f^{(n)}) \geq \epsilon$. Then
\begin{align}\label{stat:eq:underest}
  \mathbb P\left(  k_{\epsilon/2} (Y^{(n)}) = k \right) \geq 1 - 2 \exp \left(- \frac{\epsilon^2 n}{8 v + 4 \kappa \epsilon} \right) \ .
\end{align}
\end{theorem}

\begin{proof}
First suppose that $k>0$. Notice that by~\eqref{stat:def:est_nrmodes} we have that $ k_{\epsilon/2} (Y) = k$ iff $s_{k-1} (Y^{(n)}) \geq \frac \epsilon 2$ and  $s_{k} (Y^{(n)}) < \frac \epsilon 2$. Furthermore, by assumption we have that $s_{k-1}(f^{(n)}) \geq \epsilon$ and $s_{k}(f^{(n)}) =0$. Therefore,  $ k_{\epsilon/2} (Y) \neq k$ implies that 
\begin{align*}
|s_{k-1}(f^{(n)}) -s_{k-1}(Y^{(n)})|\geq\frac \epsilon 2 \quad\text{or} \quad |s_{k}(f^{(n)}) -s_{k}(Y^{(n)})|\geq\frac \epsilon 2 \ .
\end{align*}
For $k=0$, by a similar argument, we have that $ k_{\epsilon/2} (Y^{(n)}) \neq 0$ implies that $|s_{k}(f^{(n)}) -s_{k}(Y^{(n)})|\geq\frac \epsilon 2$.

Thus, for all $k\geq 0$, \cref{stat:thm:ineqbernst} implies that
\begin{align*}
 \mathbb P \left(  k_{\epsilon/2}(Y^{(n)}) \neq k \right) \leq  \mathbb P \left( \max_{j\in \NN_0} |s_{\!j\,}(Y^{(n)})- s_{\!j\,}(f^{(n)})| \geq \frac\epsilon 2 \right) \leq
 2 \exp \left(- \frac{\epsilon^2 n}{8 v + 4 \kappa \epsilon} \right) \ .
\end{align*}
\end{proof}

We stress that the bound in \cref{stat:cor:muissspec} is remarkably simple, as it depends on the signatures $s_{\!j\,}(f^{(n)})$, $j\in \{0,\ldots,k-1\}$, only through $s_{k-1}(f^{(n)})$, which in a sense represents the signature that is hardest to detect. Notice furthermore that the bound in \cref{stat:cor:muissspec} does not depend on the (unknown) number of signatures $k$.

\paragraph*{Limitations of Kolmogorov signatures}
Kolmogorov signatures are by no means suitable for all kinds of signals. Indeed, as might be expected intuitively, Kolmogorov signatures are not well suited for sparse signals that have high peaks with small support (the \emph{needle in a haystack problem}). In order to illustrate this effect, consider signals of the following kind:
\begin{align}\label{stat:subopt_Kolmogorov}
f_n(x)= \begin{cases}
 (1+\varepsilon)\sqrt{2\log n} \; \text{if} \; x\in[j/,(j+1)/n)\ ,\\
 0 \; \text{otherwise}\ ,
 \end{cases}
\end{align}
for some $\varepsilon>0$ and for some $j\in \{0,\dots, n-1\}$ \emph{that is a priori not known}. Note that there exists \emph{no} statistical testing procedure that can asymptotically (as the number of observation $n \to \infty$) detect signals with intensity as in~\eqref{stat:subopt_Kolmogorov} for $\varepsilon <0$ with positive detection power, see, e.g.,~\cite{donoho2004}. For $\varepsilon>0$, sup norm based thresholding is known to achieve the optimal detection boundary~\cite{donoho2004}. In contrast, Kolmogorov signature based thresholding at $\tau_n(\alpha)$ as described above is not able to detect signals of the type~\eqref{stat:subopt_Kolmogorov} for any $\varepsilon >0$:  
\begin{theorem}[Kolmogorov signatures and sparse signals]\label{thm:Kolmogorov-suboptimal}
Let $f_n:[0,1]\to \RR$ be as in \eqref{stat:subopt_Kolmogorov}, and let $Y_i= f_n(\frac{i}{n})+\epsilon_i$, where $\epsilon_1,\ldots,\epsilon_{n}\overset{\text{i.i.d.}}{\sim}\mathcal{N}(0,1)$. Then for any $\alpha \in (0,1)$ one has
\begin{align*}
\lim_{n \rightarrow \infty} \mathbb{P}\left(  k_{\tau_n(\alpha)} (Y^{(n)}) = 1 \right) = 0\ ,
\end{align*}
i.e., it is impossible to detect the single mode of $f_n$ when thresholding Kolmogrov signatures at $\tau_n(\alpha)$.
\end{theorem}
\begin{proof}
We have 
\begin{align*}
\mathbb{P}\left(  k_{\tau_n(\alpha)} (Y^{(n)}) = 1 \right) 
&\leq \mathbb{P}\left(  k_{\tau_n(\alpha)} (Y^{(n)}) \geq 1 \right) \\
&=\mathbb{P}\left( s_0(Y^{(n)}) \geq \tau_n(\alpha) \right) \\
&\leq\mathbb{P}\left( d_K(Y^{(n)},0) \geq \tau_n(\alpha) \right) \ ,
\end{align*}
where $d_K(Y^{(n)},0)$ denotes the Kolmogorov distance of the observations $Y^{(n)}$ to the zero function. Let $\mu_n :=(1+\varepsilon)\sqrt{2\log n}$. Then the last term can be further estimated as 
\begin{align*}
\mathbb{P}\left( d_K(Y^{(n)},0) \geq \tau_n(\alpha) \right) 
&= \mathbb{P}\left( \frac 1 n \max_{m =1, \dots, n}\left |\sum_{i=1}^m \epsilon_i + \mu_n \right|  \geq \tau_n(\alpha) \right) \\
&\leq \mathbb{P}\left( \frac 1 n \max_{m =1, \dots, n}\left |\sum_{i=1}^m \epsilon_i \right| + \frac{\mu_n}{n}  \geq \tau_n(\alpha) \right) \\
&= \mathbb{P}\left( \frac {1}{\sqrt{n}} \max_{m =1, \dots, n}\left |\sum_{i=1}^m \epsilon_i \right| + \frac{\mu_n}{\sqrt{n}}  \geq \sqrt{n}\tau_n(\alpha) \right) \ .
\end{align*}
The claim now follows from the fact that with $n\to \infty$ one has $\mu_n/\sqrt{n} \to 0$, $\sqrt{n}\tau_n(\alpha)\to \infty$, and 
 $$ \frac{1}{\sqrt{n}}  \max_{m =1, \dots, n}\left |\sum_{i=1}^m \epsilon_i \right| \overset{\mathcal D}{\rightarrow} \sup_{0\leq x\leq 1}B(x) \  ,$$ where $B$ denotes a standard Brownian motion on $[0,1]$.
\end{proof}

\subsection{Simulations using Kolmogorov signatures}
\begin{figure}[t]
\center{
 \includegraphics[width= 0.45\columnwidth]{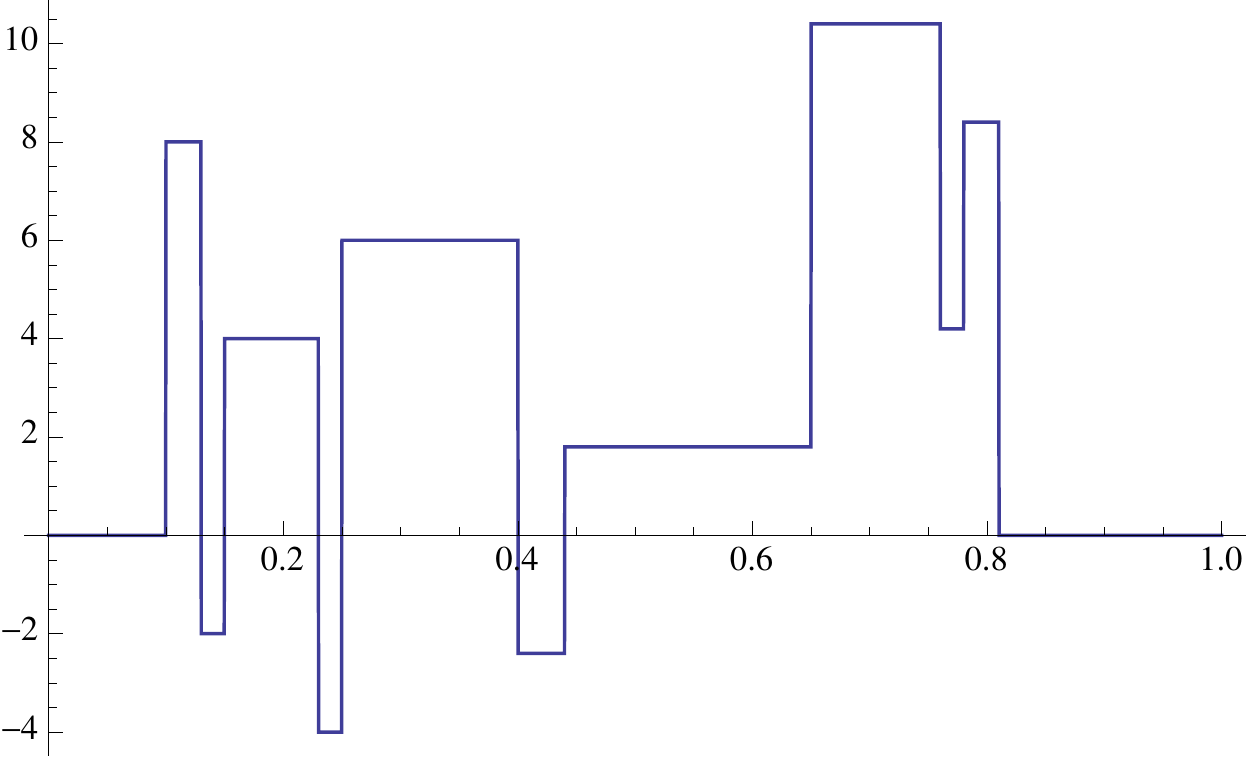}
 \qquad
 \includegraphics[width= 0.45\columnwidth]{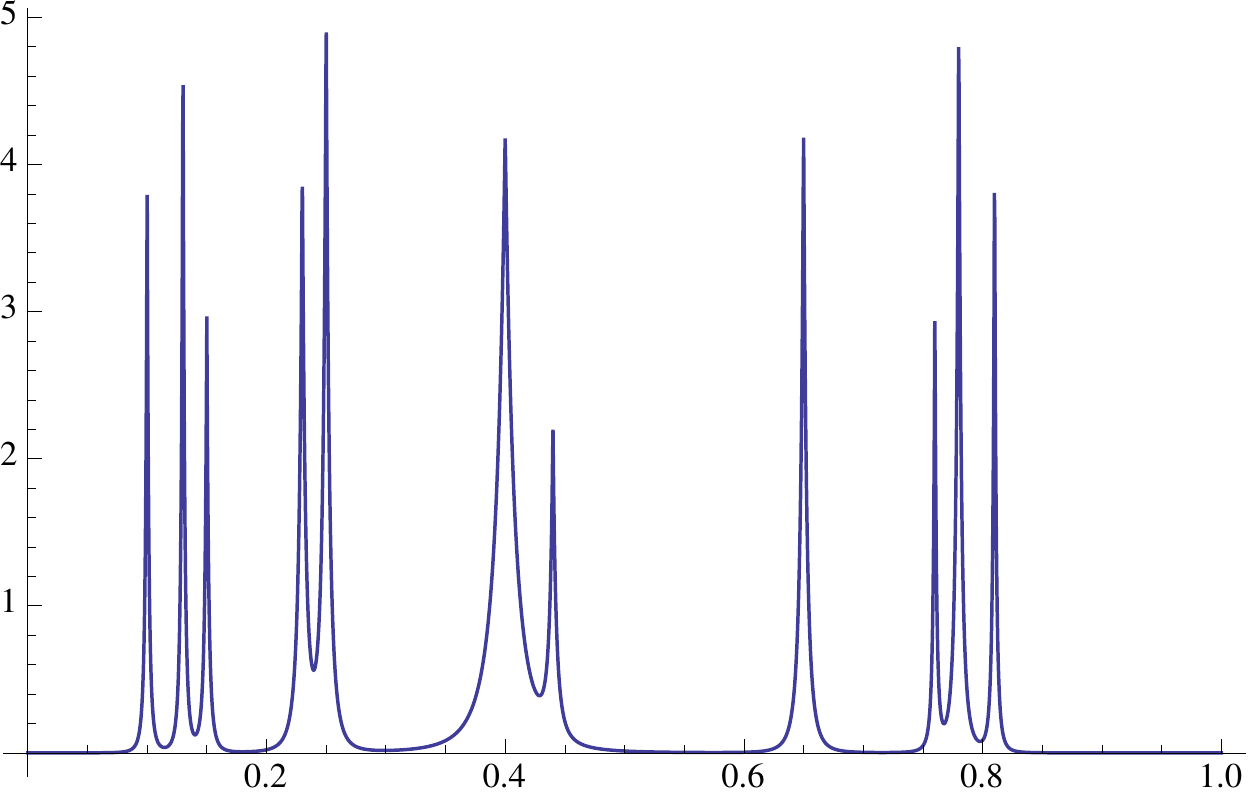}
 }
\caption{Left: signal \emph{blocks}; Right: signal \emph{bumps} \cite{Donoho1995Wavelet}.}
\label{fig:sim:signals}
\end{figure}

We illustrate the validity of our approach by means of a simulation study for the signals \emph{blocks} and \emph{bumps} \cite{Donoho1995Wavelet}, which are shown in \cref{fig:sim:signals}.
Concerning detection of modes, the two signals are of different types as they contain modes of different lengths.
For a function $f$ with $k$ modes and observations $Y$ from \eqref{intro:model} 
the theory in the previous Section shows that the number of modes $k$ can be estimated by thresholding of the 
empirical signatures. This approach clearly relies on the fact that $s_{k-1}(Y^{(n)})$ and $s_{k}(Y^{(n)})$
can be distinguished with high probability. Here, we investigate this empirically by considering the quantity
\begin{align*}
 \Delta(Y) = \frac{s_{k-1}(Y^{(n)})}{s_k(Y^{(n)})} \ .
\end{align*}
For our simulation we consider independent Gaussian noise. We note that the bound in \cref{stat:rem:gauss}
is constant for increasing $n$ if the variance is linearly increasing in $n$. This suggests that 
the expected value of $\Delta(Y)$ is also constant in this case.

We chose $\sigma= \sqrt{n}/16$ for \emph{blocks} and $\sigma =\sqrt{n}/256$ for \emph{bumps} and computed the average value 
of $\Delta(Y)$ in 1000 Monte-Carlo simulations. The results in \cref{sim:table:res} show that $\Delta(Y)$ is approximately
constant for $n\geq 1024$. Further, for both signals the ratio $\Delta(Y)$ is bounded away from $1$, which empirically confirms that
the number of modes can be estimated by thresholding.
\begin{table}[ht]
\centering
\begin{tabular}{c|c|c}
    n     & blocks  & bumps \\ \hline
    256   & 1.28726 & 2.08565 \\
    1024  & 1.57086 & 1.8708 \\
    4096  & 1.52344 & 1.85699 \\
    16384 & 1.52735 & 1.84809 \\
    65536 & 1.52647 & 1.83197 \\
\end{tabular}
\label{sim:table:res}
  \caption{Average values of $\Delta(Y)$ for \emph{blocks} and \emph{bumps} as in \cref{fig:sim:signals}. The results are obtained
  from 1000 simulations with independent Gaussian noise with $\sigma= \sqrt{n}/16$ and $\sigma =\sqrt{n}/256$ for \emph{blocks} and \emph{bumps}, respectively. As becomes evident from \cref{fig:sim:signals}, the correct number of modes of the signal is $k=5$ and $k=11$, respectively.}
\end{table}

\subsection{Sup norm based persistence signatures}

We contrast the results of the previous sections with what holds true for persistence signatures. Throughout this section, let $s_{j,\infty}$ denote the signatures with respect to the sup norm.
For simplicity we restrict our exposition
to functions with one \emph{single} mode. More precisely, we consider functions of the type
\begin{align}\label{stat:vanish_funct}
f_n(x)= \begin{cases}
 \delta_n \; \text{if} \; x\in[1/3,2/3)\ ,\\
 0 \; \text{otherwise}\ ,
 \end{cases}
\end{align}
with $\delta_n \rightarrow 0$.  It is well known that it is possible to detect the single mode of $f_n$ with probability tending to one as $n\to \infty$ if 
\begin{align}\label{stat:opt_detect_rate}
 \delta_n \sqrt{n} \rightarrow \infty \ ,
\end{align}
see, e.g., \cite{Chan11detection,vanderVaart00}. From \cref{stat:cor:muissspec} it follows, using $\epsilon = \delta_n$, that Kolmogorov signatures can correctly detect the single mode of signals in~\eqref{stat:vanish_funct-intro} by thresholding signatures at $\delta_n/2$. In contrast, for persistence signatures, there exists no thresholding strategy that can detect the single mode with probability one:

\begin{theorem}\label{stat:thm:supnorm}
Let $Y_i= f_n(i/n)+\epsilon_i$, where $\epsilon_1,\ldots,\epsilon_{n}\overset{\text{i.i.d.}}{\sim}\mathcal{N}(0,1)$, and let $f_n:[0,1]\to \RR$ be as in \eqref{stat:vanish_funct} with $\delta_n$ such that $\delta_n \sqrt{\log n} \rightarrow 0$.
For any arbitrary sequence $q_n\in \mathbb{R}$
\begin{align*}
\limsup_{n \rightarrow \infty} \mathbb{P}\left(  k^\infty _{q_n} (Y) = 1 \right) < 1\ .
\end{align*}
\end{theorem}
The proof of \cref{stat:thm:supnorm} requires some preparation.
First, recall that a sequence of random variables $Z_1,\ldots,Z_n$ follows a \emph{Gumbel extreme value limit} (GEVL)
with sequences $a_n$ and $b_n$ if
\begin{align*}
\lim_{n \rightarrow \infty} \mathbb{P}\left( \max_{1\leq i \leq n} Z_i \leq a_n + b_n x \right) = e^{-e^{-x}} .
\end{align*}
A sequence of i.i.d. standard normal random variables follows a GEVL with
\begin{align}\label{stat:def:gumbseq}
a_n=\sqrt{2 \log n} - \left(1/2 \log\log n + \log 2 \sqrt{\pi}\right)/{\sqrt{2\log n}} \ , \quad b_n= 1/\sqrt{ 2 \log n } \ .
\end{align}

Another essential ingredient of the proof of \cref{stat:thm:supnorm} is the following lemma.
\begin{lemma}\label{stat:lem_monapprox}
Let $m \in \NN$, assume $\epsilon_1,\ldots,\epsilon_{2m}\overset{\text{i.i.d.}}{\sim}\mathcal{N}(0,1)$, and set
\begin{align*}
\Delta_m= \min_{h\in\RR^{2m}:h_1\leq\ h_2\leq\dots\leq\ h_{2m}} || \epsilon - h ||_\infty \ .
\end{align*}
Then, with $a_m$ and $b_m$ as in \eqref{stat:def:gumbseq},
\begin{align*}
\lim_{m \rightarrow \infty} \mathbb{P}\left(  \Delta_m \leq a_m + b_m x \right) \leq e^{-e^{-x}}.
\end{align*}
\end{lemma}

\begin{proof}[of \cref{stat:lem_monapprox}]
Consider a fixed vector $h\in\RR^{2m}$ such that $h_1\leq h_2 \leq \dots\leq h_{2m}$. In particular, $h_j \leq h_m$ for all $j \leq m$ and $h_j \geq h_m$ for all $j \geq m$.
Let $M^{(1)}= \max_{i=1,\ldots,m} \epsilon_i$ and $M_{(2)}= \min_{i=m+1,\ldots,2m} \epsilon_i$, and observe
\begin{align*}
|| \epsilon - h ||_\infty \geq \max \left\{ M^{(1)}-h_m, h_m - M_{(2)} \right\} \  .
\end{align*}
Hence,
\begin{align*}
\Delta_m \geq \min_{\zeta \in \RR} \max \left\{ M^{(1)}-\zeta, \zeta - M_{(2)}  \right\} = \frac{1}{2}\left(M^{(1)}-M_{(2)} \right)\overset{\mathcal{D}}{=} M^{(1)},
\end{align*}
where $A \overset{\mathcal{D}}{=} B$ means that $A$ and $B$ are equally distributed.
This implies that
\begin{align}\label{stat:eq:ineqiso}
\lim_{m \rightarrow \infty}\mathbb{P}\left(  \Delta_m \leq a_m + b_m x \right) \leq \lim_{m \rightarrow \infty} \mathbb{P}\left( M^{(1)}\leq a_m + b_m x \right)= e^{-e^{-x}},
\end{align}
because $M^{(1)}$ is the maximum of $m$ independent standard normal random variables and follows a GEVL with $a_m$ and $b_m$.
\end{proof}

\begin{proof}[of \cref{stat:thm:supnorm}]
To ease notation, we assume that $n=6 m$ for some $m\in \NN$ and hence $m=m(n)$.
First, we observe that
\begin{align}\label{stat:supnorm:ineq1}
\mathbb{P} \left(  k_{q_n}^\infty(Y^{(n)}) \geq 1 \right) = \mathbb{P} \left( s_{0,\infty}(Y^{(n)}) \geq q_n \right)
\; \text{ and } \;
\mathbb{P} \left(  k_{q_n}^\infty(Y^{(n)}) > 1 \right) = \mathbb{P} \left( s_{1,\infty}(Y^{(n)}) \geq q_n \right)\ .
\end{align}
Since $s_{0,\infty}(Y^{(n)}) \leq d_\infty(f_n, Y)+ s_0(f_n)$ (by \cref{intro:lem:stab}) and $s_0(f_n) = \delta_n/2$ it holds that
\begin{align}\label{stat:supnorm:ineq11}
\mathbb{P} \left(  k_{q_n}^\infty(Y^{(n)}) \geq 1 \right) & \leq \mathbb{P}\left(d_\infty(f_n, Y) \geq q_n - \delta_n/2 \right) \\
&=\mathbb{P} \left( \frac{d_\infty(f_n, Y)-a_n}{b_n} \geq \frac{q_n-a_n}{b_n} - \frac{\delta_n}{2 b_n} \right) \nonumber \\
&= \mathbb{P} \left( \frac{d_\infty(f_n, Y)-a_n}{b_n} \geq \frac{q_n-a_n}{b_n} + o(1) \right), \nonumber
\end{align}
with $a_n$ and $b_n$ as in \eqref{stat:def:gumbseq}.
Since $d_\infty(f_n, Y)=\max_{i=1,\ldots,n} |\epsilon_i|$ it follows that for any $x\in \RR$
one has $\mathbb{P}\left( d_\infty(f_n, Y) \geq x \right) \leq 2 \mathbb{P}\left( \max_{i=1,\ldots,n} \epsilon_i \geq x \right)$ by symmetry. Therefore,
\begin{align}\label{stat:supnorm:ineq121}
\lim_{n\rightarrow \infty} \mathbb{P} \left( \frac{d_\infty(f_n, Y)-a_n}{b_n} \geq x \right) \leq 2 \left(1 - e^{-e^{-x}}\right).
\end{align}
Further, for $i=0,\ldots,5$ we define
\begin{equation*}
\Delta_i^{\pm}= \min_{h\in\RR^{m}:h_1\leq\ h_2\leq\dots\leq\ h_{m}} || \pm (Y_{im+1},\ldots,Y_{(i+1)m}) - h ||_\infty \ .
\end{equation*}
Recall that $s_{1,\infty}(Y^{(n)}) = \inf_{g \in X_{1}} d_{\infty}(g,Y)$. Observe that any $g \in X_{1}$ is either monotonically increasing or decreasing on $[i/6,(i+1)/6]$ for some $0\leq i\leq 5$. Otherwise $g$ would have two modes, which contradicts $g \in X_{1}$.
For this reason, we find $s_{1,\infty}(Y^{(n)}) \geq \min \left\{ \Delta_0^{-}, \Delta_0^{+}, \ldots,\Delta_5^{-}, \Delta_5^{+} \right\}$. Note that
$ \Delta_0^{-}, \Delta_0^{+}, \ldots,\Delta_5^{-}, \Delta_5^{+}$ are identically distributed and independent asymptotically. Therefore,
\begin{align}\label{stat:supnorm:ineq21}
\mathbb{P} \left(  k_{q_n}^\infty(Y^{(n)}) > 1 \right)
= \lim_{n \rightarrow \infty} \mathbb{P}\left( s_{1,\infty}(Y^{(n)}) \geq q_n \right) \nonumber
& \geq \lim_{n \rightarrow \infty} \mathbb{P} \left(\min \left\{ \Delta_0^{-}, \Delta_0^{+}, \ldots,\Delta_5^{-}, \Delta_5^{+} \right\} \geq q_n \right) \\ \nonumber
& =  \lim_{n \rightarrow \infty} \left(1 - \mathbb{P} \left(\Delta_0^- < q_n \right) \right)^{12}.
\end{align}
In order to prove the assertion, we show that for some $\beta\in(0,1)$
\begin{equation*}
 \lim_{n \rightarrow \infty}\mathbb{P} \left(  k_{q_n}^\infty(Y^{(n)}) \geq 1 \right) \geq 1- \beta
\end{equation*}
already implies
\begin{equation*}
 \lim_{n \rightarrow \infty}\mathbb{P} \left(  k_{q_n}^\infty(Y^{(n)}) > 1 \right) > 0
\end{equation*}
for any sequence $q_n\in \RR$.
In other words, no thresholding procedure can estimate the number of true modes $k=1$ with probability tending to one.
Combining \eqref{stat:supnorm:ineq11} and \eqref{stat:supnorm:ineq121} shows that
$\lim_{n \rightarrow \infty}\mathbb{P}\left(  k_{q_n}^\infty(Y^{(n)}) \geq 1 \right)\geq 1-\beta$
implies 
\begin{equation*}
 q_n \leq a_n+b_n z_\beta+ o(b_n)
\end{equation*}
where $z_\beta$ is defined by $2(1-\exp(-\exp(-z_\beta)))=\beta$ (it is assumed w.l.o.g. that $\beta < e^{1/6}/2$).
We then find from \eqref{stat:supnorm:ineq21} that
\begin{align*}
\lim_{n \rightarrow \infty} \mathbb{P}\left( k_{q_n}^\infty(Y^{(n)}) > 1 \right)
& \geq \lim_{n \rightarrow \infty} \left(1 -  \mathbb{P} \left(\Delta_0^- < a_n+b_n z_\beta+ o(b_n) \right) \right)^{12}\\
& = \lim_{n \rightarrow \infty} \left(1 - \mathbb{P} \left(\frac{\Delta_0^- - a_m}{b_m} < \frac{a_n-a_m}{b_m} + \frac{b_n}{b_m} z_\beta \right) \right)^{12}\\
& \geq \left( 1- e^{-e^{-z_\beta-\log 6}} \right)^{12}.
\end{align*}
Here the last inequality follows from \cref{stat:lem_monapprox} together with $\frac{b_n}{b_m}\rightarrow 1$ and $\frac{a_n - a_m}{b_m} \rightarrow \log 6$.
The proof is then completed by observing that $z_\beta + \log 6 < \infty$, which yields $\left( 1- e^{-e^{-z_\beta-\log 6}} \right) > 0$.
\end{proof}

\section{Taut strings}\label{sec:taut-strings}
In order to compute Kolmogorov signatures, we require some well known and also some less known results about \emph{taut strings}, 
see e.g. \cite{DavKov01, Mammen1997Locally}.  We prove a result that is central for our exposition and appears to be interesting in its own right: Taut strings minimize the number of critical points within a certain (quite general) class of functions.

For a given $f \in \L$ with antiderivative $F$, consider the $d_\infty$-ball $D_\alpha(F)$ of radius $\alpha\geq 0$ around $F$. We refer to $D_\alpha(F)$  as the \emph{$\alpha$-tube} around $F$. The \emph{taut string}, denoted by $F_\alpha$, is the unique function in $D_\alpha(F)$ whose graph, regarded as a curve in $\mathbb{R}^2$, has minimal total curve length, subject to boundary conditions 
\begin{align*}
F_\alpha(0)=F(0) \quad\ \text{and} \quad F_\alpha(1)=F(1) \ . 
\end{align*}
For existence and uniqueness, we refer to~\cite{Grasmair2007Equivalence, Grasmair2008Generalizations}. $F_\alpha$ is \emph{Lipschitz continuous} for all $\alpha >0$ (see~\cite{Grasmair2007Equivalence}, proof of Lemma 2); thus its derivative $f_\alpha$ (defined a.e.) is in $\L^\infty$ and we may hence choose $f_\alpha\in \L$.

Therefore, the properties that $F_\alpha \in D_\alpha(F)$ and that the graph of $F_\alpha$ has minimal curve length are equivalent to 

\begin{align*}
d_K(f,f_\alpha)\leq \alpha \quad \text{and} \quad \int_0^1 \sqrt{1+f^2_\alpha(t)} \, \mathrm d t = \min \ ,
\end{align*}
respectively.
The aim of this section is to show the following result.
\begin{figure}[t]
\center
 \includegraphics[width= 0.5\columnwidth]{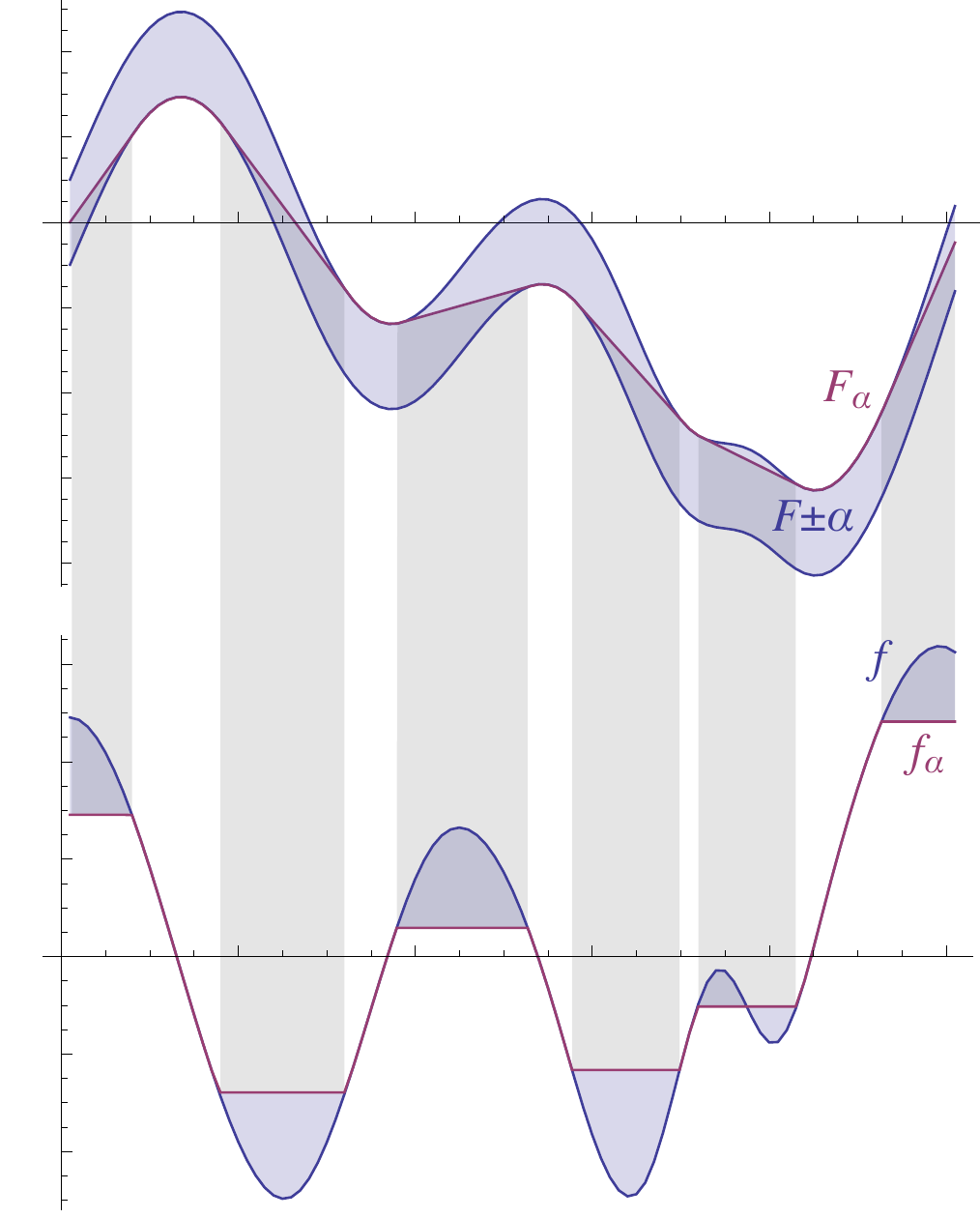}
 \caption{Taut string $F_\alpha$ (purple) in the $\alpha$-tube around $F$ (top) and its derivative $f_\alpha$ (bottom).
  }
 \label{fig:taut-string-example}
\end{figure}

\begin{theorem}\label{thm:tautStringMinModes}
For all $f\in \L$ and all $\alpha >0$, the derivative $f_\alpha\in \L$ of the taut string~$F_\alpha$ minimizes the number of modes among all function $g\in \L$ with $d_K(f,g)\leq \alpha$. 
\end{theorem} 
The proof requires some preparation. Let the top and bottom functions of the $\alpha$-tube around the antiderivative $F$ of $f\in \L$ be denoted by
\begin{align*}
T_\alpha (t) := F(t)+\alpha \quad\text{and}\quad B_\alpha (t) := F(t)-\alpha \ ,
\end{align*}
respectively. Furthermore, let 
\begin{align*}
S_{T,\alpha} =\{t\in [0,1]  :  F_\alpha (t) = T_\alpha(t) \} \quad\text{and}\quad S_{B,\alpha} =\{t\in [0,1]  :  F_\alpha (t) = B_\alpha(t) \} 
\end{align*}
denote the sets where the taut string touches the top (resp.~bottom) of the $\alpha$-tube.
\begin{lemma}[\citet{Grasmair2008Generalizations}] \label{lemma:Grasmair-new} 
For every $\alpha >0$, the taut string $F_\alpha$ is the unique function in $D_\alpha(F)$ with $F_\alpha(0)=F(0)$ and $F_\alpha(1)=F(1)$ that is convex on every connected component of $(0,1)\setminus S_{B,\alpha}$ and concave on every connected component of $(0,1)\setminus S_{T,\alpha}$. In particular, $F_\alpha$ is piecewise affine outside of $S_{B,\alpha}\cup S_{T,\alpha}$.
\end{lemma}
\Cref{lemma:Grasmair-new} gives rise to a characterization of the modes of the derivative of a taut string (see \cref{lemma:positiveInflection} below). This characterization resembles the fact that an isolated local maximum (local minimum) of $f_\alpha$ corresponds to a point (or interval) where its antiderivative $F_\alpha$ changes from being locally convex to locally concave (concave to convex), see \cref{fig:taut-string-example}.  Accordingly, we define:
\begin{definition}[maximally concave, convex, and affine intervals]\label{def:maximallyConvex}
Fix $\alpha >0$. An interval $I=[a,b]\subset [0,1]$ is called \emph{maximally affine} if $F_\alpha$ is affine on $I$ but not on any interval that properly contains $I$. An interval $I=[a,b]\subset [0,1]$ that is \emph{not} maximally affine is called \emph{maximally convex (concave)} if $F_\alpha$ is convex (concave) on $I$ but not on any interval that properly contains $I$.
\end{definition}
Observe that by \cref{lemma:Grasmair-new}, if $F_\alpha$ is not affine on all of $[0,1]$, then every $t\in [0,1]$ is contained in a maximally concave or a maximally convex interval (or possibly both). By construction, maximally convex (concave) intervals are mutually disjoint (within their respective classes).
\begin{definition}[positive and negative inflection intervals]\label{def:inflection}
Fix $\alpha >0$. 
An interval $I=[a,b]\subset (0,1)$ is called a \emph{positive} (\emph{negative}) \emph{inflection interval} of~$F_\alpha$ if $I$~is a maximally affine interval of~$F_\alpha$ and $F_\alpha$~is 
convex (concave) on some non empty neighborhood of~$a$ and concave (convex) on some non empty neighborhood of~$b$.
\end{definition}
Notice that we deliberately require that $a>0$ and $b<1$ in our definition of inflection intervals. As a direct consequence of \cref{lemma:Grasmair-new} we obtain:
\begin{lemma}
\label{lem:inflectionSwich}
Fix $\alpha >0$. 
If $[a,b]$ is a positive inflection interval of $F_\alpha$, then
 $F_\alpha(a) = T_\alpha(a)$ and $F_\alpha(b) = B_\alpha(b)$; if it is a negative inflection interval, then $F_\alpha(a) = B_\alpha(a)$ and $F_\alpha(b) = T_\alpha(a)$.
\end{lemma}
Moreover we have:
\begin{lemma}\label{lemma:maximallyConvex}
Fix $\alpha >0$. Then~$F_\alpha$ has the following properties:
\begin{enumerate}[(i)]
\item The number of maximally convex, the number of maximally concave, and the number of inflection intervals of~$F_\alpha$  is finite.
\item Maximally convex and maximally concave intervals are interleaved, i.e., the set of points between two consecutive maximally convex (concave) intervals belongs to a maximally concave (convex) interval.
\item The intersection of a maximally convex (concave) with an immediately consecutive maximally concave (convex) interval is a positive (negative) inflection interval, and every inflection interval arises in this way.
\end{enumerate}
\end{lemma}
\begin{proof}
Let $T_\alpha$ and $B_\alpha$ denote the top and bottom of the $\alpha$-tube around $F_\alpha$, respectively. Since $F_\alpha$ is continuous, the graphs of $T_\alpha$ and $B_\alpha$ are compact sets. Let $I$ be a maximally concave, a maximally convex, or an inflection interval of $F_\alpha$. By \cref{def:maximallyConvex,def:inflection,lemma:Grasmair-new}, the graph of $F_\alpha$ restricted to $I$ must then contain an affine segment that connects $T_\alpha$ with $B_\alpha$ (or  $B_\alpha$ with $T_\alpha$). Therefore, the arc length of the graph of $F_\alpha$, restricted to $I$, is bounded from below by the Euclidean distance $d_\alpha$ between the graphs of $B_\alpha$ and  $T_\alpha$. Since these sets are compact and disjoint, one has $d_\alpha >0$. Since $d_\alpha$ is independent of $I$, and since $F_\alpha$ is Lipschitz, it follows that the length of $I$ is bounded from below by a number that only depends on $\alpha$ and the Lipschitz constant of $F_\alpha$. Hence, since 
maximally convex (concave) intervals are mutually disjoint, there can only exist finitely many of them. Likewise, since positive (negative) inflection intervals are disjoint, there can only exist finitely many of those. Properties (i) and (ii) are then a straightforward consequence of \cref{lemma:Grasmair-new}. 
\end{proof}
The next lemma states the promised characterization of the modes of the derivative of a taut string.
\begin{lemma}\label{lemma:positiveInflection}
Fix $\alpha >0$ and define
\begin{align*}
f_\alpha(t) = 
 \lim_{\epsilon \to 0}\inf_{0<\delta<\epsilon} \frac{F_\alpha(t+\delta)-F_\alpha(t-\delta)}{2\delta} \quad \text{if} \quad 0<t<1
\end{align*}
and $f_\alpha(t) = \lim_{s\to t}f_\alpha(s)$ for $t\in \{0,1\}$. Then the number of positive inflection intervals of $F_\alpha$ equals the number of modes of $f_\alpha$, and this number is finite. 
\end{lemma}
\begin{proof}
First notice that the definition of $f_\alpha(0)$ and $f_\alpha(1)$ is meaningful since $F_\alpha$ is affine in some neighborhood of $0$ and $1$. 

If $F_\alpha$ is affine on all of $[0,1]$, then there is nothing to show. So suppose that this is not the case. Consider a finite partition $P=\{t_0, \dots, t_{|P|}\}$ of $[0,1]$. 
Notice that $f_\alpha$ is nowhere decreasing (nowhere increasing) on intervals where $F_\alpha$ is convex (concave). Hence, for $t_i$ to count a mode of $f_\alpha$, i.e., $\peaks(f_\alpha,P,i)=1$, the pair $(t_{i-1}, t_i)$ must not belong to the same maximally concave interval and the pair $(t_i, t_{i+1})$ must not belong to the same maximally convex interval of $F_\alpha$. Since, by assumption, $F_\alpha$ is not affine on all of $[0,1]$, every $t\in [0,1]$ belongs to a maximally concave or maximally convex interval (or both). Therefore, by property (i) of \cref{lemma:maximallyConvex}, to each mode of $f_\alpha$ counted by $P$ there corresponds at least one change from a maximally convex to an immediately consecutive maximally concave interval. By property (ii) of \cref{lemma:maximallyConvex}, the total number of such changes is equal to the number of positive inflection intervals, which we denote by $\I^+(F_\alpha)$. It follows that $\I^+[F_\alpha]\geq \peaks(f_\alpha)$. 

Vice versa, by considering a partition of $[0,1]$ such that there exists (apart from $t_0=0$ and $t_{|P|}=1$) exactly one point in each positive and each negative inflection interval, it is straightforward to show that $\I^+[F_\alpha]\leq \peaks(f_\alpha)$. 

Finally, finiteness of $\peaks(f_\alpha)$ follows from the fact that there are only finitely many positive inflection intervals.
\end{proof}
With these preparations, we are now in the position to prove \cref{thm:tautStringMinModes}.
\begin{proof}[of \cref{thm:tautStringMinModes}]
Let $g\in \L$ with antiderivative $G$ such that $d_K(f,g)\leq \alpha$. Consider a positive inflection interval $[a,b]$ of $F_\alpha$. By \cref{lem:inflectionSwich}, $F_\alpha(a) = T_\alpha(a)$, $F_\alpha(b) = B_\alpha(b)$, and $F_\alpha$ is affine on $[a,b]$. In particular, $G(a) \leq F_\alpha(a)$ and $G(b) \geq F_\alpha(b)$, and thus
\begin{align*}
f_\alpha(t) = \frac{F_\alpha(b) - F_\alpha(a)}{b-a} \leq \frac{G(b) - G(a)}{b-a} \quad\text{for all $t \in (a,b)$} \ .
\end{align*}
For every Lebesgue-integrable $g:[a,b]\to\RR$ with $G(t)=G(a)+\int_a^t g(s)\, \mathrm d s$, there exist sets $C_1, C_2 \subset [a,b]$ of positive Lebesgue measure such that 
\begin{align*}
g(c_1) \leq \frac{G(b)-G(a)}{b-a} \leq g(c_2)
\end{align*}
for all $c_1\in C_1$ and all $c_2\in C_2$. 

Hence, for every positive inflection interval $[a,b]$ there exists $t\in (a,b)$ such that $g(t) \geq f_\alpha(t)$.  By a similar argument, for every negative inflection interval $[a,b]$ there exists $t\in (a,b)$ such that $g(t) \leq f_\alpha(t)$. 
By \cref{lemma:positiveInflection}, whenever $\peaks(f_\alpha) >0$ (otherwise there is nothing to show), the set of positive inflection intervals of $F_\alpha$ is not empty. Therefore, one can choose a partition $P=\{t_0, \dots, t_{|P|}\}$ of $[0,1]$ that contains (apart from $t_0=0$ and $t_{|P|}=1$) exactly one point in the interior of each inflection interval of $F_\alpha$ such that $g(t_i) \geq f_\alpha(t_i)$ whenever $t_i$ lies in a positive inflection interval and $g(t_i) \leq f_\alpha(t_i)$ whenever $t_i$ lies in a negative inflection interval. By the proof of \cref{lemma:positiveInflection}, $\peaks(f_\alpha)=\peaks(f_\alpha,P)$ for \emph{any} partition $P$ that contains (apart from $t_0=0$ and $t_{|P|}=1$) exactly one point in the interior of each inflection interval. Such partitions $P$ count a mode of $f_\alpha$ precisely for every \emph{positive} inflection interval of $F_\alpha$. Since positive and negative inflection intervals are interleaved and their interiors are disjoint, we obtain 
that $\peaks(g,P) \geq \peaks(f_\alpha,P)=\peaks(f_\alpha)$. Thus $\peaks(g) \geq \peaks(f_\alpha)$.
\end{proof}

\section{Computing Kolmogorov signatures}

The results of the previous section lead to an efficient algorithm for computing Kolmogorov signatures. Let $X \subset \L$ be some subset, and let $f\in X$ with antiderivative $F$. Suppose that $X$ contains the derivatives $f_\alpha$ of the taut stings $F_\alpha$ for all $\alpha \geq 0$. For example, let $X$ be the space of piecewise constant functions. For $\alpha$ large enough, $F_\alpha$ is affine on all of $[0,1]$, and its derivative~$f_\alpha$ has no modes. If $f$ has any modes at all, then by lowering $\alpha$ continuously, $F_\alpha$ will at some point develop a positive inflection interval below some threshold $\alpha_0 >0$. By \cref{thm:tautStringMinModes} and \cref{lemma:positiveInflection}, the value of $\alpha_0$ is precisely the distance of $f$ to the set of functions in X with zero modes, i.e., $s_0(f) = \alpha_0$. Continuing this way, and defining $\alpha_k$ as the smallest $\alpha$ for which $f_\alpha$ has at most $k$ modes, one finds that $s_k(f) = \alpha_k$ for all $k$.

The idea of the algorithm below is to reverse this observation: Starting from $f= f_0$, we incrementally compute the 
values of $\alpha$ (in \emph{increasing} order) at which the number of modes of~$f_\alpha$ decreases. To this end, we work with the space $X$ of piecewise constant functions on a fixed partition $0=t_0 <t_1 < \dots < t_n =1$ of $[0,1]$. Notice that since we require $X\subset \L$, we have $f(t_i) = \frac 1 2 (f|_{(t_{i-1},t_i)}+f|_{(t_{i},t_{i+1})})$ for all non-boundary points $t_i$ of the partition.

Our starting point is a reformulation of \cref{lemma:Grasmair-new} for piecewise constant functions.
\begin{lemma}
\label{lem:tautStringVertices}
Let $f$ be a piecewise constant function with antiderivative $F$. Then the taut string $F_\alpha$ is the unique continuous piecewise linear function in $D_\alpha(F)$ with $F_\alpha(0)=F(0)$ and $F_\alpha(1)=F(1)$ such that if $t$ is an increasing (decreasing) discontinuity of $f_\alpha=F_\alpha '$, then $F_\alpha(t)=F(t)+\alpha$ ($F_\alpha(t)=F(t)-\alpha$).
\end{lemma}

Fix $\alpha \geq 0$. Let $I=(a,b)\subseteq (0,1)$ be an open interval, and let $f_\alpha$ be constant on $(a,b)$. We call $I$ \emph{regular} for $f_\alpha$ if either $a=0$ and $b=1$ or $a>0$, $b< 1$, and there exists $\epsilon >0$ such that for all $0< \delta \leq \epsilon$ either $f_\alpha(a) > f_\alpha(a-\delta)$ and $f_\alpha(b) < f_\alpha(b +\delta)$ or $f_\alpha(a) < f_\alpha(a-\delta)$ and $f_\alpha(b) > f_\alpha(b +\delta)$. We call $I=(a,b)$ \emph{maximal (respectively~minimal)} for $f_\alpha$ if $a>0$, $b<1$, and there exists $\epsilon >0$ such that for all $0< \delta \leq \epsilon$ one has $f_\alpha(a) > f_\alpha(a-\delta)$ and $f_\alpha(b) > f_\alpha(b +\delta)$  (respectively~$f_\alpha(a) < f_\alpha(a-\delta)$ and $f_\alpha(b) < f_\alpha(b +\delta)$). We call $I$ \emph{critical} if it is minimal or maximal. Finally, we call $I=(a,b)$ a \emph{boundary interval} for $f_\alpha$ if either $a=0$ and $b<1$ or $a>0$ and $b=1$, and $(a,b)$ is the largest such interval on which $f_\alpha$ is constant. As 
a consequence of \cref{lem:tautStringVertices} we  obtain:

\begin{corollary}
\label{cor:valuesOfTautString}
Away from discontinuities,
$f_\alpha$ has the following form: either 
\begin{itemize}
\item
$t$ lies on a regular interval $I=(a,b)$ of $f_\alpha$ with value $f_\alpha(t)=%
\frac{F(b)-F(a)}{b-a}$, 
\item $t$ lies on a locally minimal/maximal interval $I=(a,b)$ of $f_\alpha$ with value $f_\alpha(t)=%
\frac{F(b)-F(a)\pm2\alpha}{b-a}$, %
or
\item $t$ lies on a boundary interval of $f_\alpha$ with value $f_\alpha(t)=%
\frac{F(b)-F(a)\pm\alpha}{b-a}$. %
\end{itemize}
\end{corollary}

This corollary is central for our computation of Kolmogorov signatures. First observe that a value of a maximal interval is continuously decreasing with growing $\alpha$, the value of a minimal interval is continuously increasing, and the value of a regular interval remains unchanged. Moreover, if $\alpha$ is increased only slightly, then the discontinuities of $f_\alpha$ remain unchanged; indeed:

\begin{lemma}
\label{lem:jumpsStayConstant}
Let $F$ be piecewise linear. For every $\alpha \geq 0$ there is $\delta>0$ such that the points of discontinuity of $f_\beta$ coincide with those of $f_\alpha$ for all $\beta$ with $\alpha \leq \beta < \alpha + \delta$. Moreover, if $t$ lies on a regular interval of $f_\alpha$, then $f_\beta(t)=f_\alpha(t)$. 
\end{lemma}
\begin{proof}
Define $G_\beta$ by the properties of \cref{lem:tautStringVertices}, using the discontinuities of $f_\alpha$, i.e., if $t$ is an increasing (decreasing) discontinuity of $f_\alpha$, then define $G_\beta(t):=F(t)+\beta$ (resp.~$G_\beta(t):=F(t)-\beta$); set $G_\beta(0):=F(0)$ and $G_\beta(1):=F(1)$, and interpolate linearly. Then $\|G_\beta-F_\alpha\|_\infty=\beta-\alpha$ and thus, since $F_\alpha \in D_\alpha(F)$, we have that $\|G_\beta-F\|_\infty \leq \beta$, i.e., $G_\beta \in D_\beta(F)$. For $\delta$ sufficiently small, the discontinuities of $g_\beta=G_\beta '$ have the same type as those of $f_\alpha$. But since $F_\beta$ is uniquely defined by the properties of \cref{lem:tautStringVertices} with respect to these discontinuities, we must have $F_\beta=G_\beta$.
\end{proof}

As a consequence, for every $\alpha \geq 0$, there exists a minimal number $\mu(\alpha)>\alpha$ such that $f_\beta$ and $f_\alpha$ have the same points of discontinuity for all $\beta$ with $\alpha \leq \beta < \mu(\alpha)$ but the set of points of discontinuity of $f_{\mu(\alpha)}$ is different from that of $f_\alpha$. We call $\mu(\alpha)$ the \emph{merge value} of $\alpha$. The merge value is the smallest number strictly greater than $\alpha$ for which a critical interval or a boundary interval of $f_{\mu(\alpha)}$ reaches the value of an adjacent constant interval, and the corresponding discontinuity vanishes.  Each discontinuity of $f_\alpha$ that is incident to a critical or a boundary interval is a possible candidate for such an event. Consider such a discontinuity $b$ between two consecutive constant intervals $I=[a,b]$ and $J=[b,c]$ of $f_\alpha$. For an interval $I=[a,b]$, let $F_I:=F(b)-F(a)$. As a consequence of \cref{cor:valuesOfTautString} and \cref{lem:jumpsStayConstant}, we 
obtain that the merge value $\mu(\alpha)$ is the smallest number among all merge value \emph{candidates} $m_{I,J}$ of $f_\alpha$, which are computed as follows:

If $I$ is critical and $J$ is regular or vice-versa, then the merge value candidate is
\[
m_{I,J} = %
\frac12 \left| F_I - \frac{|I|}{|J|}F_J\right|.
\]
If both $I$ and $J$ are critical, then the merge value candidate is
\[
m_{I,J} = %
\left|\frac{|I|F_J-|J|F_I}{2(|I|+|J|)}\right|.
\]
If $I$ is critical and $J$ is a boundary interval, then the merge value candidate is
\[
m_{I,J} = %
\left|\frac{|I|F_J-|J|F_I}{|I|+2|J|}\right|.
\]
If $I$ is a boundary interval and $J$ is critical, then the merge value candidate is
\[
m_{I,J} = %
\left|\frac{|I|F_J-|J|F_I}{2|I|+|J|}\right|.
\]
If $I$ is a boundary interval and $J$ is regular or vice-versa, then the merge value candidate is
\[
m_{I,J} = %
\left| F_I - \frac{|I|}{|J|}F_J\right|.
\]
If both $I$ and $J$ are boundary intervals, then the merge value candidate is
\[
m_{I,J} = %
\left |\frac{|I|F_J-|J|F_I}{|I|+|J|}\right|.
\]

We define the sequence of merge values $\mu_1<  \mu_2 < \mu_3<\dots$ of $f$ as follows. Starting from $\mu_1 := \mu(0)$, let $\mu_{i+1} := \mu({\mu_i})$. By construction, the values $\alpha = \mu_i$ are precisely those values where the number of discontinuities of $f_{\alpha}$ decreases with increasing $\alpha$. 

Observe that the merge value \emph{candidates} of $f_{\mu_{i+1}}$ are equal to those of $f_{\mu_i}$ \emph{except only} for the merged intervals $I$ and $J$, i.e., those intervals that have the same value for $f_{\mu_{i+1}}$ but did not have the same value for $f_{\mu_i}$. This suggests an efficient way for computing Kolmogorov signatures of~$f$ in reverse order.
Starting with $\alpha=0$, we iterate in increasing order through the sequence of merge values of $f$. 
In a min-priority queue, we maintain the merge value \emph{candidates} $m_{I,J}$. In each iteration $i$, the lowest merge value candidate is the next value $\mu_i$.
Upon a merge, the corresponding discontinuity is removed, and the merge value candidates of the neighboring discontinuities are recomputed and updated in the priority queue. The discontinuities are organized in a linked list to allow fast access to the neighbors. If the number of modes of $f_\alpha$ has decreased upon a merge, the value $\alpha$ is prepended to the sequence of computed signatures. This can only occur if one of the merged intervals is maximal.
The method is summarized in pseudocode in \cref{alg:kolmogorovSign}.
Using an appropriate heap data structure, the running time is $O(n\log n)$, where $n$ is the number of function values of $f$.

\begin{algorithm}[ht]
\caption{Computing Kolmogorov signatures}
\label{alg:kolmogorovSign}
\begin{algorithmic}[1]
\Procedure {KolmogorovSequence}{$f$: list of function values}
  \State $\alpha=0$
  \State $S=\text{empty sequence}$

  \State $L=$ jumps of $f$ (linked list)
  
  \State $Q=$ merge values of the jumps (priority queue)
  
  \While{the priority queue $Q$ is not empty}
    \State let $\alpha$ be the smallest merge value in $Q$
    \State let $I=[a,b]$ and $J=[b,c]$ be the corresponding intervals
    \If{$I$ and $J$ are minimum/maximum or boundary/maximum of $f_\alpha$}
     \State prepend $\alpha$ to $S$
    \EndIf
    \State remove $b$ from the list $L$ of discontinuities
    \State remove $\alpha$ from the priority queue $Q$
    \State recompute merge values of $a$ and $c$ and update priority queue $Q$
    \EndWhile
  \State \Return $S$ 
\EndProcedure
\end{algorithmic}
\end{algorithm}

\section*{Acknowledgements}
We would like to thank the anonymous reviewers for their very helpful suggestions for revising our manuscript and Carola Schoenlieb for inspiring discussions.

\bibliography{baueru}

\end{document}